\documentclass[12pt,a4paper,reqno]{amsart}
\usepackage{vmargin,color}
\usepackage[latin1]{inputenc}
\usepackage{amsmath,amsfonts,amsthm,epsfig}

\def\R{\mathbb R}
\def\N{\mathbb N}

\def\cal{\mathcal}
\def\H{{\cal H}}

\def\E{{\cal E}}

\def\dd{{\rm d}}

\def\a{\alpha}
\def\b{\beta}
\def\g{\gamma}
\def\de{\delta}
\def\e{\varepsilon}
\def\l{\lambda}
\def\om{\omega}
\def\s{\sigma}

\def\pa{\partial}

\def\k{\kappa}

\def\big{\bigskip}

\newtheorem{theorem}{Theorem}[section]
\newtheorem{lemma}[theorem]{Lemma}

\newtheorem{remark}{Remark}[section]

\numberwithin{equation}{section}
\numberwithin{figure}{section}

\begin{document}

\title{Stability for the Brunn-Minkowski and Riesz rearrangement inequalities, with applications
to Gaussian concentration and finite range non-local isoperimetry}

\author{E. A. Carlen$^1$ \& F. Maggi$^2$\vspace{20pt}\\
\small{$1.$ Department of Mathematics, Rutgers University,}\\[-6pt]
\vspace{8pt}\small{ 110 Frelinghuysen Road,
Piscataway NJ 08854-8019}\\
\vspace{1pt}\small{$2.$ Department of Mathematics, University of Texas, Austin}\\[-6pt]
\small{2515 Speedway, Austin, TX 78712}\\
}

\begin{abstract}
{\rm We provide a simple, general argument to obtain improvements of concentration-type inequalities starting from improvements of their corresponding isoperimetric-type inequalities. We apply this argument to obtain robust improvements of the Brunn-Minkowski inequality (for Minkowski sums between generic sets and convex sets) and of the Gaussian concentration inequality. The former inequality is then used to obtain a robust improvement of the Riesz rearrangement inequality under certain natural conditions. These conditions are compatible with the applications to a finite-range nonlocal isoperimetric problem arising in statistical mechanics.}
\end{abstract}

\maketitle

\footnotetext
[1]{Work partially
supported by U.S. National Science Foundation
grant DMS-1201354.    }
\footnotetext
[2]{Work partially
supported by U.S. National Science Foundation
grants DMS 1265910 and DMS 1361122.\\
\copyright\, 2015 by the authors. This paper may be reproduced, in its
entirety, for non-commercial purposes.}

\tableofcontents

\section{Introduction} In this paper we present a general argument to deduce robust improvements of the Brunn-Minkowski inequality and of the Gaussian concentration inequality starting from the corresponding quantitative isoperimetric inequalities. We then exploit the former result  to obtain a robust improvement of the Riesz rearrangement inequality in the case of a strictly decreasing interaction kernel that acts on nested sets. Finally, we discuss
how this last result may be applied to provide a quantitative geometric description of near-minimizing droplets for the Gates-Lebowitz-Penrose free energy, a problem arising in statistical mechanics that motivated this research.

\subsection{Stability for the Brunn-Minkowski inequality}
If $E$ and $F$ are Lebesgue measurable sets in $\R^d$, $E+F=\{x+y:x\in E\,,y\in F\}$ is their Minkowski sum, and $|G|$ denotes the (outer) Lebesgue measure of a set $G\subset\R^d$, then the Brunn-Minkowski inequality ensures that
\begin{equation}
  \label{brunn-minkowski inq}
  |E+F|^{1/d}\ge |E|^{1/d}+|F|^{1/d}\,.
\end{equation}
Henstock and Macbeath \cite{henstockmacbeath} proved that if $0<|E||F|<\infty$, then equality holds in \eqref{brunn-minkowski inq} if and only if $E$ and $F$ are equivalent to their convex hulls, which in turn, up to translations, are homothetic to each other. A natural question is then how to relate the size of the gap between the left-hand side and the right-hand side of \eqref{brunn-minkowski inq} to the distance of (suitably scaled and translated copies of) $E$ and $F$ from a suitably chosen convex set. This problem has been solved in the case that both $E$ and $F$ are convex sets in \cite{figallimaggipratelliBrunnMink}. In this case, it was shown that
\begin{equation}
  \label{brunn-minkowski inq improved on convex}
  |E+F|^{1/d}\ge \left(|E|^{1/d}+|F|^{1/d}\right)\,\left\{1+\frac{\a(E;F)^2}{C(d)\,\s(E;F)^{1/d}}\right\}\,,
\end{equation}
where $\s(E;F)=\max\{|E|/|F|,|F|/|E|\}$, and where $\a(E;F)$ is defined as
\begin{equation}
\label{assymetry}
\a(E;F)=\frac12\inf\left\{\frac{|E\Delta(x_0+r\,F)|}{|E|}:x_0\in\R^d\,,r^d=\frac{|E|}{|F|}\right\}\,.
\end{equation}
(The factor $1/2$ is included so to have $\a(E;F)\in[0,1)$). We shall find it convenient to restate \eqref{brunn-minkowski inq improved on convex} as
\begin{equation}
  \label{brunn-minkowski inq improved on convex 2}
  C(d)\,\de(E;F)\ge \a(E;F)^2\,,
\end{equation}
where  we have set
\begin{equation}
  \label{bm deficit}
  \de(E;F)=\s(E;F)^{1/d}\Big\{\frac{|E+F|^{1/d}}{|E|^{1/d}+|F|^{1/d}}-1\Big\}\,.
\end{equation}
The advantage of formulation \eqref{brunn-minkowski inq improved on convex 2} of \eqref{brunn-minkowski inq improved on convex} is that $\de(E;F)$ and $\a(E;F)$ are both scale invariant quantities, meaning that
\[
\de(\l\,E;\l\,F)=\de(E;F)\,,\qquad \a(\l\,E;\mu\,F)=\a(E;F)\,,\qquad\forall\l\,,\mu>0\,.
\]
(Note that, in general, if $\l\ne \mu$ then $\de(\l\,E;\mu\,F)$ may differ from $\de(E;F)$.) Our first main result is a quantitative improvement of \eqref{brunn-minkowski inq} in the spirit of \eqref{brunn-minkowski inq improved on convex} in the case that one of two sets $E$ and $F$ is a convex set with positive measure. In the following we shall thus fix $K$ to be an open, bounded, convex set in $\R^d$ containing the origin.

\begin{theorem}\label{thm: BM 1/4}
  For every $d\ge1$ there exists a positive constant $C(d)$ with the following property. If $E\subset\R^d$ is a Lebesgue measurable set with $0<|E|<\infty$, then
  \begin{eqnarray}
    \label{brunn-minkowski inq improved Br 1}
  &&\a(E;K)\le C(1)\,\de(E;K)\,,\hspace{5 cm}\mbox{if $d=1$}\,;
  \\
    \label{brunn-minkowski inq improved Br 2}
  &&\a(E;K)^4\le C(d)\,\max\Big\{1,\frac{|K|}{|E|}\Big\}^m\,\de(E;K)\,,\hspace{2 cm}\mbox{if $d\ge2$}\,.
  \end{eqnarray}
  Here we can take $C(1)=2$, $m=(4d+2)/d$, and
  \[
        C(d)=\frac{d}{\log(2)}\,\Big(\frac{\sqrt{2^{5\,d}\,C_0(d)}}{d}+d\,2^{d-1}+2^d\Big)^4\,,\qquad\mbox{if $d\ge2$}\,,
  \]
  where $C_0(d)$ is defined in \eqref{C0n} below.
\end{theorem}

\begin{remark}
  {\rm The estimate \eqref{brunn-minkowski inq improved Br 1} for the one-dimensional case $d=1$ is sharp in the decay rate of $\a(E;K)$ as $\de(E;K)\to 0$. On the other hand, if $E$ is a convex set with $0<|E|<\infty$, then by \eqref{brunn-minkowski inq improved on convex} we have $\a(E;K)^2\le C(d)\de(E;K)$ for every $d\ge2$, and this inequality is sharp in the decay rate of $\a(E;K)$ as $\de(E;K)\to 0$, see \cite[Section 4]{FigalliMaggiPratelliINVENTIONES}. It is thus natural to conjecture that the power $4$ on the left-hand side of \eqref{brunn-minkowski inq improved Br 2} should be replaced by the power $2$. Moreover, it should be possible to remove the corrective factor $\max\{1,|K|/|E|\}^m$ on the right-hand side of \eqref{brunn-minkowski inq improved Br 2}.} In any case, in the application of this result to be discussed here, we shall have $|K| < |E|$, in which case  the factor is $1$.
\end{remark}

\begin{remark}
  {\rm Improvements of the Brunn-Minkowski inequality \eqref{brunn-minkowski inq} in the general case when none of the two sets is assumed to be convex have been recently obtained by Figalli and Jerison in \cite{figallijerison1,figallijerison2}. For example, in \cite{figallijerison1} it is shown that if $E\subset\R^d$ and $|E+E|\le |2E|(1+\de(d))$, then, with ${\rm co}(E)$ denoting the convex hull of $E$,
  \[
  c(d)\,\Big(\frac{|{\rm co}\,(E)\setminus E|}{|E|}\Big)^{8\cdot 16^{d-1}\cdot d!\cdot (d-1)!}\le \frac{|E+E|}{|2E|}-1\,,
  \]
  where $\de(d)$ and $c(d)$ are positive computable constants.}
\end{remark}

To prove Theorem \ref{thm: BM 1/4}, we exploit known quantitative improvements of the (Wulff) isoperimetric inequality (associated to the convex set $K$) through the use of the coarea formula. Precisely, given an open bounded convex set $K$ containing the origin, one sets
\[
\|\nu\|_*=\sup\Big\{x\cdot\nu:x\in K\Big\}\,,\qquad\nu\in S^{n-1}\,,
\]
and correspondingly introduces a notion of {\em anisotropic perimeter} by setting
\[
P_K(E)=\int_{\pa E}\|\nu_E\|_*\,{\rm d}\H^{d-1}=\limsup_{r\to 0^+}\frac{|E+r\,K|-|E|}r\,,
\]
in case $E$ is an open set with Lipschitz boundary in $\R^d$. The most important case is that in which $K = B=\{x\in\R^d:|x|<1\}$, in which case $\|\cdot\|_*$
is simply the Euclidian norm and we obtain the (usual) {\em perimeter}
$$P(E) := P_B(E) = \H^{d-1}(\partial E)=\limsup_{r\to 0^+}\frac{|E+r\,B|-|E|}r\,\ .$$
(Here $\H^s$ stands for the $s$-dimensional Hausdorff measure on $\R^d$, and $r\,F=\{r\,x:x\in F\}$. We shall also set $s\,B=B_s$ and $B_{x,s}=x+B_s$ for every $x\in\R^d$ and $s>0$.)

It is well known that the Brunn-Minkowski inequality implies the Wulff inequality
\begin{equation}
  \label{wulff inequality}
  P_K(E)\ge d\,|K|^{1/d}\,|E|^{(d-1)/d}\,,\qquad\mbox{$0<|E|<\infty$}\,,
\end{equation}
where equality holds if and only if $|E\Delta (x+r\,K)|=0$ for some $x\in\R^d$ and $r>0$. In \cite{FigalliMaggiPratelliINVENTIONES} a quantitative improvement of \eqref{wulff inequality} was proved, in the form
\begin{equation}
    \label{quantitative isop inq}
    P_K(E)\ge d|K|^{1/d}|E|^{(d-1)/d}\Big(1+\frac{\a(E;K)^2}{C_0(d)}\Big)\,,\qquad\mbox{if $0<|E|<\infty$}\,,
\end{equation}
where
\begin{equation}
  \label{C0n}
C_0(d)=\frac{181\,d^7}{4\,(2-2^{1-(1/d)})^{3/2}}\,.
\end{equation}
(See \cite{fuscomaggipratelli} for the case $K=B$ of \eqref{quantitative isop inq}.) Our starting point in the proof of Theorem \ref{thm: BM 1/4} is then the remark that, if $|E|=|K|$ and $r>0$, then by the coarea formula (and provided $E$ is closed, see Lemma \ref{lemma: coarea})
\begin{eqnarray}\label{coint}
  |E+r\,K|-(|E|^{1/d}+|r\,K|^{1/d})^d&=&|E+r\,K|-|K+r\,K|\nonumber\\
  &=&\int_0^r P_K(E+s\,K)-P_K(K+s\,K)\,{\rm d}s\,.\hspace{1cm}
\end{eqnarray}
The integrand here is positive for every $s\in(0,r)$: indeed, $|E+s\,K|\ge |K+s\,K|$ by the Brunn-Minkowski inequality, and thus
$P_K(E+s\,K)\ge P_K(K+s\,K)$ by the Wulff inequality \eqref{wulff inequality}. If, instead of the Wulff inequality \eqref{wulff inequality} one applies its improved form \eqref{quantitative isop inq}, then one gets
\begin{eqnarray}\label{coint star}
  &&|E+r\,K|-(|E|^{1/d}+|r\,K|^{1/d})^d
  \\\nonumber
  &&\ge d\,|K|^{1/d}\,\int_0^r |E+s\,K|^{(d-1)/d}\frac{\a(E+s\,K;K)^2}{C_0(d)}\,{\rm d}s\,.
\end{eqnarray}

The main difficulty in proving Theorem~\ref{thm: BM 1/4} is that   $\a(E+s\,K;K)$ may decrease to zero very rapidly as $s$ increases:  For example, suppose $K = B$, and $E$ is the ball of radius $2$ that has been ``perforated'' by removing a large number of small disjoint balls of radius at most $\epsilon$ from the interior -- think of a Swiss cheese with many tiny holes. We can arrange this construction so that $|E|=|B|$. Then, for $s>\epsilon$, one has $\a(E+s\,B;B) = 0$, while
$$
\delta(E,s\,B)=\frac{1}{s^d}\,\left(\frac{2+s}{1+s} -1\right)\ .
$$
Thus, while ``Swiss cheese'' sets $E$ are such that $\alpha(E+sB;B)$ can go to zero rapidly as $s$ increases away from
zero, such sets have a large Brunn-Minkowski deficit. The proof of Theorem~\ref{thm: BM 1/4} that we give turns on showing that if $\alpha(E+sB;B)$ is much smaller than $\alpha(E;B)$ for small $s$, then $\delta(E,sB)$ is sizable for small $s$.

The main idea may be obscured by the details in the proof given in Section 2, and so we provide here a sketch of a proof for the special case $K = B$, $|E| = |B|$  and $d \geq 2$.  In this special case, we easily deduce from the definition of $\de(E;B)$, \eqref{coint star},
$|E+r\,B|\ge|B|$, and H\"older inequality, that
\begin{equation}
  \label{q1}
  C(d)\sqrt{\delta(E;B)}\ge \int_0^1 \alpha(E+r\,B;B)\,{\rm d} r\,.
\end{equation}
Next, by (the elementary) Lemma \ref{lemma: lip alpha} below, one finds that
\begin{equation}
  \label{elementary}
  |\a(E;B)-\a(F;B)|\le\frac{2\,|E\Delta F|}{\max\{|E|,|F|\}}\,,
\end{equation}
for every $E, F\subset\R^d$ with positive and finite Lebesgue measure.
Before applying this with $F=E+r\,B$, we first pick $\epsilon\in (0,1)$ and restrict the domain of integration from $r\in(0,1)$ to $r\in(0,\epsilon)$,
and then absorb a factor of $|E+rB|^{-1}\le |E|^{-1}=|B|^{-1}$ into $C(d)$, to obtain
that
\begin{equation}\label{q3}
C(d)\Big\{  \sqrt{\delta(E;B)}+\int_0^\epsilon|E\Delta (E+r\,B)|{\rm d}r \Big\}\ge \int_0^\epsilon\alpha(E;B){\rm d} r\ge\epsilon\,\alpha(E;B)\,.
\end{equation}
The key step is to bound ${\displaystyle \int_0^\epsilon|E\Delta (E+r\,B)|{\rm d}r}$ in terms of $\delta(E;B)$
and $\epsilon$. Note that
\begin{eqnarray*}
  \int_0^\epsilon|E\Delta (E+r\,B)|{\rm d} r&=&\int_0^\epsilon\Big(|E+r\,B|-|E|\Big){\rm d} r
  \\
  &=&\int_0^\epsilon {\rm d} r\int_0^r P(E+t\,B)\,{\rm d} t\,,
\end{eqnarray*}
where, again by the integration formula (\ref{coint}) and by $P(E+t\,B)\ge P(B+t\,B)$,
\begin{eqnarray*}
  \int_0^r P(E+t\,B)\,{\rm d}t&=&
  \int_0^r \Big(P(E+t\,B)-P(B+t\,B)\Big)\,{\rm d}t
  +\int_0^r P(B+t\,B)\,{\rm d}t
  \\
  &\le&|E+B| - |B+B|
+   C(d)\,r\le C(d)(\delta(E)+r)\,.
\end{eqnarray*}
Integrating over $r\in(0,\epsilon)$ we eventually prove
\[
\epsilon\,\alpha(E;B)\le C(d)\Big\{\sqrt{\delta(E;B)}+\epsilon\,\delta(E;B)+\epsilon^2\Big\}\,,
\]
and then optimize the choice of $\epsilon$ by setting $\epsilon=\delta(E;B)^{1/4}$.

\subsection{Improvements of the Gaussian concentration inequality} The strategy for proving Theorem~\ref{thm: BM 1/4} that we have just described is  applicable in other situations. We illustrate this  by considering the Gaussian concentration inequality.  Let us denote by $\g_d$ the Gaussian measure on $\R^d$, so that
\[
\g_d(E)=\frac1{(2\pi)^{d/2}}\int_Ee^{-|x|^2/2}\,{\rm d}x\,,\qquad E\subset\R^d\,.
\]
Given $\nu\in S^{d-1}$ and $s\in\R$ we set $H_\nu(s)=\{x\in\R^d\ : x\cdot\nu<s\}$, $H(s)=H_{e_1}(s)$,
\begin{equation}
  \label{phi definition}
  \phi(s)=\g_d(H_\nu(s))=\frac1{\sqrt{2\pi}}\,\int_{-\infty}^se^{-z^2/2}\,\dd z\,,
\end{equation}
and for every $E\subset\R^d$ we let $s_E\in\R$ be such that
\[
\g_d(E)=\phi(s_E)\,.
\]
With this notation, the Gaussian concentration inequality says that
\begin{equation}
  \label{gaussian concentration}
  \g_d(E+rB)\ge\g_d(H(s_E)+rB)\,,\qquad\forall r>0\,,
\end{equation}
with equality if and only if $E=H_\nu(s_E)$ for some $\nu\in S^{d-1}$. We now want to improve this inequality into a quantitative statement, and we shall do this by exploiting Gaussian isoperimetry. Let us recall that given an open set $E$ with Lipschitz boundary, the quantity
\[
P_\g(E)=\frac1{(2\pi)^{(d-1)/2}}\,\int_{\pa E}e^{-|x|^2/2}\,{\rm d}\H^{d-1}
=\limsup_{r\to 0^+}\frac{\g_d(E+rB)-\g_d(E)}r\,,
\]
is  the Gaussian perimeter of $E$, and we have the Gaussian isoperimetric inequality,
\begin{equation}
  \label{gaussian isoperimetric inequality}
  P_\g(E)\ge P_\g(H(s_E))\,,
\end{equation}
with equality if and only if $E=H_\nu(s_E)$ for some $\nu\in S^{d-1}$.

An important point of contrast with the Wulff  inequality (\ref{wulff inequality}) is that while $P_K(rK) = d|K|r^{(d-1)/d}$
is monotone increasing in $r$, $P_\gamma(H(s))$ is not monotone in $s$. In fact,
$\lim_{s\to\pm\infty}P_\gamma(H(s)) = 0$.

The quantitative analysis of \eqref{gaussian isoperimetric inequality} was initiated in \cite{cianchifuscomaggipratelliGAUSS,mossel} using a natural Gaussian analog of $\alpha(E;F)$
defined by
\[
\a_\g(E)=\inf_{\nu\in S^{d-1}}\,\g_d(E\Delta H_{\nu}(s_E))\,.
\]
The best result to date is that
\begin{equation}
  \label{gauss improved1}
  P_\g(E)-P_\g(H(s_E))\ge\frac{e^{s_E^2/2}}{c\,(1+s_E^2)}\,\a_\g(E)^2\,,\qquad c=80\,\pi^2\,\sqrt{2\pi}\,,
\end{equation}
proved in \cite{barchiesibrancolinijulin}.

A key property of \eqref{gauss improved1} is that it is dimension independent,
and we definitely desire that this strong property  be reflected in a quantitative version of \eqref{gaussian concentration}. To this end, given $E\subset\R^d$ and $r>0$ we set
\begin{equation}
  \label{deficit gaussian}
  \de_\g^r(E)=\max\Big\{1,\frac1r\Big\}\,\sup_{0<t<r}\,\g_d(E+B_t)-\g_d(H(s_E)+B_t)\,,\qquad r>0\,.
\end{equation}
Notice that the factor $1/r$ is needed if $r$ is very small, because in that regime one needs to consider an isoperimetric type deficit. The same feature appears in the Euclidean case, see \eqref{delta r}. The  supremum over $t\in (0,r)$ in the definition of the deficit in necessary because of the non-monotonicity of $P_\gamma(H(s))$
as a function of $s$, as noted above.

Next, given $\l\in(\g_d(E),1)$, we define
\begin{equation}
  \label{gaussian ray}
  r_{E}(\l)=\sup\Big\{r>0:\g_d(E+rB)<\l\Big\}\,.
\end{equation}
With this notation in force, we have the following theorem.

\begin{theorem}
  \label{thm gaussian concentration}
  Given $E\subset\R^d$ with $\g_d(E)<1$ and $\l\in(\g_d(E),1)$, one has
  \begin{equation}\label{gaussian concentration quantitative}
    \a_\g(E)^4\le C_*(\l)\,\de_\g^{r_E(\l)}(E)\,,
  \end{equation}
  where, by definition,
  \[
  C_*(v)=(5+1280\,\pi^3)^2\,\,(1+\phi^{-1}(v))^2\,,\qquad\forall v\in(0,1)\,.
  \]
\end{theorem}

\begin{remark}
  {\rm Notice that \eqref{gaussian concentration quantitative} degenerates as we allow $\l\to 1^-$.}
\end{remark}

The relation between \eqref{gaussian concentration} and \eqref{gaussian isoperimetric inequality} is similar -- but not entirely analogous -- to that existing between the Brunn-Minkowski inequality \eqref{brunn-minkowski inq} (with $F=K$ convex) and the Wulff inequality \eqref{wulff inequality}. Indeed, we can still write the deficit in \eqref{gaussian concentration} as an integral of Gaussian isoperimetric deficits, so that  \eqref{coint} now takes the form
\begin{equation}
  \label{gauss intro}
  \g_d(E+rB)-\g_d(H(s_E)+rB)=\frac1{\sqrt{2\pi}}\,\int_0^r\,P_\g(E+tB)-P_\g(H(s_E)+tB)\,dt\,.
\end{equation}
However, now we cannot infer the non-negativity of the integrand by isoperimetry (compare with the argument below \eqref{coint}), because of the non-monotonicity of  $P_\gamma(H(s))$  in $s$. Indeed, if we consider the decomposition
\begin{equation}
  \label{montpellier}
  \begin{split}
    P_\g(E+B_t)-P_\g(H(s_E)+B_t)&=P_\g(E+B_t)-P_\g(H(s_{E+B_t}))
    \\
    &+P_\g(H(s_{E+B_t}))-P_\g(H(s_E)+B_t)\,,
  \end{split}
\end{equation}
then the first term in the sum on the right-hand side is non-negative by \eqref{gaussian isoperimetric inequality}, while the sign of second term depends on the values of $\g_d(E)$ and $\g_d(E+B_t)$. In particular, it is not clear if the left-hand side of \eqref{gauss intro} is increasing in $r$, in contrast to  the Euclidean case. The possible lack of this monotonicity property is the ultimate reason for including the supremum over $t\in(0,r)$ in the definition \eqref{deficit gaussian} of $\de_\g^r(E)$.

\subsection{A finite range non-local perimeter functional}

We shall apply Theorem~\ref{thm: BM 1/4} to a finite range non-local perimeter functional that arises in statistical mechanics. In mathematical terms, our main result is Theorem~\ref{reiszst1}, a quantitative version of the Riesz rearrangement inequality in a case that is relevant to statistical mechanics. We now  briefly discuss the variational problem that motivates Theorem~\ref{reiszst1}.

Let $\Lambda$ denote the $d$ dimensional torus with period $L$, and hence volume $L^d$. For smooth functions $m$
on $\Lambda$, the van der Waals free energy functional is
\[
\mathcal{F}(m) = \int_\Lambda W(m(x)){\rm d}x + \frac{\theta}{2} \int_\Lambda |\nabla m(x)|^2{\rm d} x\
\]
where $W(m) = \tfrac14 m^2(1-m)^2$.  The function $m(x)$ specifies the mixture of two ``phases'' (think liquid and vapor, for example) at $x$, so that where $m(x) =1$,
the system is in one phase, and where $m(x) = 0$ it is in the other (and thus $m(x) \in (0,1)$ corresponds to some mixture of the phases).

Let $n\in (0,1)$,
and consider the problem of determining
\begin{equation}\label{mmmin}
\inf\left \{\ \mathcal{F}(m)\ :\ \int_\Lambda m(x){\rm d}x = nL^d\ \right\}
\end{equation}
For $\theta =0$, the problem is trivial. Let $D$ be any measurable subset of $\Lambda$ with $|D| = nL^d$, and define
$$m(x) = \begin{cases} 1 & x\in D\\ 0 & x \notin D \end{cases}\ .$$
Any such function is a minimizer. We may think of $D$ as a ``droplet'' of the $m=1$ phase in a sea of the $m=0$ phase.
For $\theta =0$, the shape of the droplet is irrelevant.

For $\theta>0$, surface tension plays a role and tries to minimize the perimeter of the droplet. A classic argument of Modica and Mortolla, that we now briefly sketch,   shows how isoperimetry comes into play. Use the co-area formula, and then the arithmetic-geometric mean inequality  to write
\begin{eqnarray}\label{sri4}
 \mathcal{F}(m) &=& \int_\R \int_{\{m = h\}}\left(
  \frac{1}{4}  \frac{(1-h^2)^2}{|\nabla m(x)|}  + \frac{\theta^2}{2}|\nabla m(x)|\right) \dd{\mathcal H}^{d-1} {\rm d}h\nonumber\\
  &\geq&  \int_\R
\frac{\theta}{\sqrt{2}}|1-h^2|
{\mathcal H}^{d-1}(\{m = h\}){\rm d}h\ \nonumber
\end{eqnarray}
where  ${\mathcal H}^{d-1}$ is $d-1$  dimensional Hausdorff measure.  It is possible to nearly saturate the
arithmetic-geometric mean inequality  by choosing $m$ to cross the boundary between the phases with a certain profile,
and then, to nearly minimize $\mathcal{F}$, the quantitative  isoperimetric inequality forces the phase boundary
to be nearly spherical -- at least when $n$ is small enough that the droplet cannot wrap around the torus.  Thus,
near minimizers of the van der Waals free energy functional, which by the rules of statistical mechanics are what one is likely to observe in equilibrium, are ``round droplets''. There is a cost to any departure from this optimal shape that is determined through the quantitative isoperimetric inequality.

The van der Waals free energy function is purely phenomenological; it cannot be derived from any underlying particle system. However, other free energy functionals, such as the Gates-Penrose-Lebowitz free energy functional
\cite{lebowitzpenrose,gatespenrose},
 do arise from particle systems, and are therefore more physically significant. While they have a similar structure, the gradient term in $\mathcal{F}$
is replaced by a finite range  non-local interaction functional that we now describe.

Let $J:[0,\infty)\to[0,\infty)$ be a  decreasing Lipschitz function supported in $[0,1]$ such that
$$\int_{\R^d}J(|x|){\rm d}x =1.$$ On square integrable functions $m(x)$ on $\R^d$, we define the functional $\mathcal{P}_J$ by setting
\begin{equation}\label{nlp}
\mathcal{P}_J(m)   = \int_{\Lambda \times \Lambda} J(|x-y|)|m(x) - m(y)|^2 {\rm d}x {\rm d}y\ .
\end{equation}
If one replaces  the gradient term in $\mathcal{F}(m)$ by $\mathcal{P}_J(m)$, one obtains a variant of
the Gates-Penrose-Lebowitz free energy functional  \cite{lebowitzpenrose,gatespenrose}:
\[
\mathcal{G}(m) = \int_\Lambda W(m(x)){\rm d}x + \mathcal{P}_J(m)\ .
\]
(The actual GPL functional has a different ``double well'' potential  function $W$ in it, but this does not matter here.)
We would like to
solve  the  minimization problem (\ref{mmmin}) with $\mathcal{G}$ in place of $\mathcal{F}$.

The functional $\mathcal{P}_J(m)$ can be thought of as a finite range non-local perimeter functional in the following sense:
Let $m$ be the characteristic function of a set $D$ of with $|D| = nL^d$, $n\in (0,1)$.
Should the boundary of $D$ be smooth enough (with a graphicality scale much larger than the interaction range of $J$), we would then have
\begin{equation}
  \label{order Pj}
  \mathcal{P}_J(m) \asymp \H^{d-1}(\partial D)\ .
\end{equation}
This motivates the intuition that $\mathcal{P}_J$ is a non-local perimeter functional, and suggests that
as for the van der Waals free energy functionals, near minimizers for $\mathcal{G}$ will necessarily be   ``droplets'' $D$ that are almost spherical, at least when $n$ is small enough that the droplets cannot wrap around the torus.

However, the Modica-Mortola strategy cannot be directly applied to the   functional $\mathcal{G}$ since the absence of
gradients prevents one from making the same argument with the co-area formula. What we do instead is to
investigate the behavior of $\mathcal{P}_J$ under spherically symmetric decreasing rearrangements. The first thing we do is to specialize to the case in which $m$ is supported in a set whose diameter is less than $L$, in which case
we may extend $m$, and the integration in (\ref{nlp}) to all of $\R^d$. (See \cite{CarlenCELM} for the reduction to this case  in the statistical mechanics problem.) We then have the functional
\begin{equation}
\mathcal{P}_J(m)   = \int_{\R^d \times \R^d} J(|x-y|)|m(x) - m(y)|^2 {\rm d}x {\rm d}y\ ,
\end{equation}
and are in a position to apply rearrangement inequalities.

\if false

This functional arises in statistical mechanics, and in phase segregation problems in particular \cite{lebowitzpenrose,gatespenrose}. In this context, the function $m(x)$ gives the mixture of two ``phases'' (think liquid and vapor, for example) at $x$, so that where $m(x) =1$,
the system is in one phase, and where $m(x) = 0$ it is in the other (and thus $m(x) \in (0,1)$ corresponds to some mixture of the phases). Indeed, a variant of the Gates-Lebowitz-Penrose free energy \cite{lebowitzpenrose,gatespenrose} consists of the sum of $\mathcal{P}_J$ and a ``double well'' potential with strict
minima at $0$ and $1$. Because of the statistical nature of the model one is interested in {\em near minimizers} of this free energy functional, and not only {\em exact minimizers}. Under suitable side conditions, the interaction between  $\mathcal{P}_J$ and the double well potential will force a phase segregation: roughly speaking, every near minimizer $m$ satisfies
\begin{equation}
  \label{ccelm}
  \bigg|\Big|\Big\{m>1-\frac1{L^\a}\Big\}\Big|-L\,|B|\bigg|\le\e\,L\,|B|,\qquad \Big|\Big\{\frac1{L^\a}\le m\le 1-\frac1{L^\a}\Big\}\Big|\le\e\,L\,|B|\,,
\end{equation}
where $L$ is large with respect to the (unitary) interaction range of $J$, $\a$ depends on the dimension $d$, and $\e\in(0,1)$ is an arbitrarily small positive constant -- the smaller $\e$, the nearer $m$ has to be an exact minimizer for \eqref{ccelm} to hold. Of course \eqref{ccelm} is only a quite rough approximation of the exact result proved in \cite[Theorem 2.3]{CarlenCELM}, and it will serve here to the sole purpose of illustrating the kind of improvement of the Riesz's rearrangement inequality needed in this context.

Although \eqref{ccelm} gives us the orders of magnitude of the phase transition described by $m$, it says nothing about the geometry of the resulting droplet. For the sake of illustration, let us set
\[
\e=o\Big(\frac1{L}\Big)\,,\qquad E=\Big\{m>1-\frac1{L^\a}\Big\}\,,\qquad F=\Big\{m>\frac1{L^\a}\Big\}\,,
\]
so that $E\subset F$ and \eqref{ccelm} implies, say,
\begin{equation}
  \label{hpx 2}
  \frac{|E|}{|F|}=1+o\Big(\frac1{L}\Big)\,,\qquad |E|\ge  \frac{L}2\,|B|\,.
\end{equation}

This expectation is completely justified by looking at the theory of symmetric decreasing rearrangements.

\fi

\subsection{Riesz rearrangement and Lieb's Theorem}

If $E$ is a measurable subset of $\R^d$ with finite measure, then we let $E^*$ denote the ball in $\R^d$ centered at $0$
with $|E^*| = |E|$. If $f$ is a non-negative function  on $\R^d$ such that for each $\lambda \geq  0$, $|\{ f  > \lambda\}| < \infty$,
the symmetric decreasing rearrangement of $f$ is the function $f^*$  given by
\begin{equation}\label{fstdef}
f^*(x) = \int_0^\infty 1_{\{ f  > \lambda\}^*}(x) \dd \lambda\ .
\end{equation}
By construction, $f^*$ is measurable, and for all $\lambda>0$,
$|\{ f^*  > \lambda\}| = |\{ f  > \lambda\}|$,
and so for any non-negative function $G$ on $\R_+$,
\[\int_{\R^d}G(f^*(x))\dd x = \int_{\R^d}G(f(x))\dd x\ .
\]
In particular, the double well potential energy term in the GPL free energy $\mathcal{G}$  is conserved in passing from $m$ to $m^*$, as it takes the form $\int_{\R^d}W(m(x))\dd x$. The interaction energy $\mathcal{P}_J(m)$ is instead decreased, as a consequence of the following deep theorem about symmetric decreasing rearrangements (for a proof, and for more discussion of rearrangements, see \cite{HLPBOOK,LiebLossBOOK}):

\begin{theorem}[Riesz rearrangement inequality]\label{rri} Let $f$, $g$ and $h$ be non-negative integrable functions
on $\R^d$.  Then
$$\int_{\R^d}\int_{\R^d}f(x) g(x-y)h(y)\dd x \dd y \leq \int_{\R^d}\int_{\R^d}f^*(x) g^*(x-y)h^*(y)\dd x \dd y\ .$$
\end{theorem}

To apply this to the functional $\mathcal{P}_J$, note that since $\int_{\R^d}J(|x|){\rm d}x =1$,
\[
\int_{\R^d\times \R^d}J(|x-y|)m(x)^2{\rm d}x {\rm d}y = \int_{\R^d} m^2 = \int_{\R^d} (m^*)^2 \ ,
\]
and thus
\begin{equation}\label{bcc1}
\mathcal{P}_J(m) - \mathcal{P}_J(m^*) = 2\,\Big(\mathcal{I}_J(m^*) - \mathcal{I}_J(m)\Big)\,,
\end{equation}
where
\begin{equation}\label{bcc2}
\mathcal{I}_J(m) = \int_{\R^d\times \R^d}m(x) J(|x-y|)m(y){\rm d}x{\rm d}y\ .
\end{equation}
Thus Theorem \ref{rri} implies that $\mathcal{P}_J(m)-\mathcal{P}_J(m^*)\ge 0$. In particular, if $m$ is a minimizer, then equality holds in the Riesz inequality with $f=h=m$ and $g=J$. Should this necessary condition for optimality imply that $m$ is radially decreasing, then one could try to perturb it in order to infer that near minimizers are nearly spherical.

Generally speaking, the cases of equality in the Riesz rearrangement inequality have been fully determined by Burchard \cite{Burchard}. The matter is quite complex, as there are many ways that equality can hold without $f$, $g$ and $h$ being translates of their rearrangements.
For example, suppose that $f = 1_F$, $g= 1_G$ and $h = 1_H$ where $F$, $G$ and $H$ are Borel sets of finite measure.
Define $A = \{y\in\R^d :y + G \subset F \}$. Then, with $\star$ denoting convolution, $1_G \star 1_F(y) = |G|$ everywhere
on $A$.  Then if $H \subset A$, $H^* \subset A^*$, and there will be equality in Riesz's inequality regardless of the ``shapes'' of $F$, $G$ and $H$.

Things are different, however, when one of the functions involved, say $g$, is symmetric decreasing and, in addition, {\em every ball} (centered at the origin) is a super-level set of $g$. The following theorem is due to Lieb \cite{LiebChoquard}:

\begin{theorem}[Lieb's theorem on cases of equality in the Riesz rearrangement inequality]\label{lrri}
Let $f$, $g$ and $h$ be non-negative integrable functions
on $\R^d$.  Suppose that $g = g^*$, and that for every $r>0$, there is a $\lambda_r> 0$ so that
$$\{ g > \lambda_r\} = rB\ .$$
Then whenever
\begin{equation}\label{lieb4}
\int_{\R^d}\int_{\R^d}f(x) g(x-y)h(y)\dd x \dd y = \int_{\R^d}\int_{\R^d}f^*(x) g^*(x-y)h^*(y)\dd x \dd y\ ,
\end{equation}
there is an $a\in \R^d$ so that $f(x) = f^*(x-a)$ and $h(y) = h^*(y -a)$ almost everywhere in $x$ and $y$.
\end{theorem}

Lieb's proof of this theorem was by induction on the dimension.  A different proof, based on the Brunn-Minkowski inequality, will allow us to make two extensions: The first is relatively simple and could be done within the framework of Lieb's proof: We relax the requirement  that {\em every} centered ball is a super level set of $g$. The more significant extension is a quantitative version asserting, roughly speaking,  that when  (\ref{lieb4}) holds with near equality instead of equality,  then $f$ and $h$
are still {\em nearly} translates of their rearrangements. The quantitative version of the Brunn-Minkowski inequality  proved
in this paper
is basis of this.  To explain the connection between this inequality and Lieb's Theorem,  we now
sketch a proof of Theorem \ref{lrri} based on the Brunn-Minkowski inequality.

The starting point is the layer-cake representation, see \cite{LiebLossBOOK},
\begin{equation}\label{layercake}
\int_{\R^d} f\star g (x)h(x) \dd x = \int_0^\infty \dd r  \int_0^\infty \dd s \int_0^\infty \dd t\, \int_{\R^d} 1_{F_r}\star1_{G_s}(x) 1_{H_t}(x)\dd x\,,
\end{equation}
where for $f$, $g$ and $h$ as above and $r,s,t>0$ we have set
\[
F_r = \{f > r\} \ ,\qquad G_s =\{g > s \}\quad{\rm and}\quad H_t = \{h > s\}\,.
\]
For each fixed $r,s$, the continuous function $1_{F_r}\star1_{G_s}(x)$ is supported on the closure of the Minkowski sum
$F_r+G_s$. (One must be careful about sets of measure zero as explained in section 4.)  One way
 to prove Lieb's theorem is to prove that for fixed $r,t>0$, there is a set $A\subset \R_+$ of strictly positive measure
 such that when $t\in A$,
 \begin{equation}\label{qqA}
 \int_{\R^d} 1_{F_r}\star1_{G_s}(x) 1_{H_t}(x)\dd x < \int_{\R^d} 1_{F_r^*}\star1_{G_s^*}(x) 1_{H_t}(x)\dd x\,.
 \end{equation}
 (Here, since $g= g^*$, we have $G_t = G_t^*$.)

 Without loss of generality, we may suppose that $|F_r| <| H_t|$.  Note that $F_r^*$ is a ball of radius $(|F_r|/|B|)^{1/d}$,
 $G_s^*$ is a ball of radius $(|G_s|/|B|)^{1/d}$ and $H_t^*$ is a ball of radius $(|H_t|/|B|)^{1/d}$.
By hypothesis, there exists an $s$ such that
 \[
  \left(\frac{|F_r|}{|B|} \right)^{1/d} + \left(\frac{|G_s|}{|B|} \right)^{1/d} =  \left(\frac{|H_t|}{|B|} \right)^{1/d}
 \]
 Then $1_{F^*_r}\star 1_{G^*_s}$  is supported in a ball of radius $(|F_r|/|B|)^{1/d} + (|G_s|/|B|)^{1/d}$, which
 is the radius of $H^*_t$. Hence
 \[
 1_{F^*_r}\star 1_{G^*_s} (x)1_{H_t^*}(x) = 1_{F^*_r}\star 1_{G^*_s}(x) \ ,
 \]
 and consequently,
 \begin{equation}\label{qqB}
 \int_{\R^d} 1_{F_r^*}\star1_{G_s^*}(x) 1_{H_t^*}(x)\dd x = \int_{\R^d} 1_{F_r^*}\star1_{G_s^*}(x) \dd x =|F_r| |G_s|\ .
  \end{equation}
However, if $F_r$ is not a ball, the Brunn-Minkowski inequality says that the support of $1_{F_r}\star 1_{G_s}$
has a measure that is strictly larger than that of $H_t$. Hence the set
\[
\{ \ x \notin H_t\quad{\rm and}\quad 1_{F_r}\star1_{G_s}(x) > 0\ \}
\]
has positive measure. Therefore
\begin{equation}\label{qqB2}
 \int_{\R^d} 1_{F_r}\star1_{G_s}(x) 1_{H_t}(x)\dd x  = |F_r||G_s| - \int_{\R^d\backslash H_t} 1_{F_r}\star1_{G_s}(x)\dd x \ .
  \end{equation}
Comparing (\ref{qqB}) and (\ref{qqB2}), we see that
\[
\int_{\R^d} 1_{F_r}\star1_{G_s}(x) 1_{H_t}(x)\dd x  < \int_{\R^d} 1_{F_r^*}\star1_{G_s^*}(x) 1_{H_t^*}(x)\dd x
\]
when $F_r$ is not (equivalent to) a ball. By a dominated convergence argument
 this strict inequality remains valid when $s$ is replaced by $s' \in [s,s+a]$ for some $a>0$.
 From here it is easy to prove Theorem~\ref{lrri}.

 The two features of the this proof that are relevant to us are the following: (1) it is ``localizable'' in $t$, $r$ and $s$, in the sense that if we consider $r$ and $s$ lying in some interval, then we only need values of $t$ such that $|G_t|$ matches $(|E_r|^{1/d}+ |F_s|^{1/d})^d$, and not any arbitrary positive number; (2) it is based on the Brunn-Minkowski inequality, for which we have a  quantitative  improvement.

 The main difficulty to be overcome in proving a quantitative version is that while  the quantitative Brunn-Minkowski inequality
 gives us an estimate on the measure of the set $(F_r + G_s)\cap H_t^c$, it is evident from (\ref{qqB2})  that
 what we really need is a lower bound on
 \begin{equation}\label{dustpaprt}
 \int_{\R^d\backslash H_t} 1_{F_r}\star1_{G_s}(x)\dd x\ .
 \end{equation}
 Even if $(F_r + G_s)\cap H_t^c$ has a large measure,  $1_{F_r}\star1_{G_s}$ may well be  small on this set,
 and then the integral may be small.

 Indeed, if $G_s$ is a ball of radius $\rho$, then $1_{F_r}\star1_{G_s}(x) = F_s \cap B_\rho(x)$, and so if $F_r$
 is the union of many small and well separated components -- think of a cloud of dust -- then
 $\|1_{F_r}\star1_{G_s}\|_\infty$ will be very small.   However, in this case, we will be far from having equality in the Riesz rearrangement inequality.

 This suggests making a decomposition of any set $E$ (here, $F_r$), as follows:
 Given $\l,\tau>0$ we set
\begin{equation}\label{EfulldefZ}
E^{\lambda,\tau} = E \backslash D^{\lambda,\tau} \ ,
\end{equation}
\begin{equation}\label{EbaddefZ}
D^{\lambda,\tau} = \Big\{ \ x\in E\ :\ \frac{|E\cap B_{x,\tau}|}{|B_{x,\tau}|} < \lambda \Big\}\ .
\end{equation}
For small $\lambda$, and any $\tau$, $D^{\lambda,\tau}$ is the ``dusty'' component of $E$.  The key to obtaining
a lower bound  on the integral in (\ref{dustpaprt}) is to show that for appropriately chosen $\lambda$ and $\tau$,
the dusty component of $F_r$ must be  very small  whenever
\[
\int_{\R^d} 1_{F_r}\star1_{G_s}(x) 1_{H_t}(x)\dd x  \approx  \int_{\R^d} 1_{F_r^*}\star1_{G_s^*}(x) 1_{H_t^*}(x)\dd x\ .
\]

This proof of Lieb's Theorem via the Brunn-Minkowski inequality also shows that the heart of the matter is a geometric inequality for super level sets. Indeed, consider a function
 $m$ with values in $[0,1]$, and notice that
 if $E_t = \{m> t \}$, then $m(x) = \int_0^1 1_{E_t}(x){\rm d}t$ and
\begin{equation}
  \label{layercake x}
  \mathcal{I}_J(m)  =\int_0^1 \int_0^1 \mathcal{E}_J(E_t,E_s){\rm d}s{\rm d}t\,,
\end{equation}
where we have set
\begin{equation}\label{Edef}
\mathcal{E}_J(E,F) = \int_{E}\int_{F} J(|x-y|){\rm d}x{\rm d}y\,,\qquad E,F\subset\R^d\,.
\end{equation}
Then defining  $\delta_J(E,F)$ by
\begin{equation}\label{redef}
\delta_J(E,F)= \mathcal{E}_J(E^*,F^*) - \mathcal{E}_J(E,F)\ ,
\end{equation}
we obtain
\[
\mathcal{I}_J(m^*)-\mathcal{I}_J(m)=\int_0^1\dd t\int_0^1 \de_J(E_t,E_s)\,\dd s\,.
\]

Notice also that for $s< t$, $E_t \subset E_s$, so we are interested in $\mathcal{E}_J(E,F)$ when $E \subset F$.
Under mild conditions of the distribution function of $m$, for each $t$ one
can bound from below the length of the interval of those values of $s$, with $s<t$, such that
\[
 \frac14\le \Big(\frac{|E_s|}{|B|}\Big)^{1/d}-\Big(\frac{|E_t|}{|B|}\Big)^{1/d}\le \frac{3}4\,,
\]
and in the statistical mechanical application, we are chiefly interested in ``droplets'' that are large compared with the
unit ball.  The following theorem is our quantitative version of Lieb's Theorem.  The application to statistical mechanics that
motivates it will be made elsewhere.

\begin{theorem}\label{reiszst1}
Let us consider a decreasing Lipschitz function $J:[0,\infty)\to[0,\infty)$  with  ${\rm spt}(J) \subset [0,1]$ such that
\begin{equation}
  \label{hp on J intro}
  \int_{\R^d}J(|x|){\rm d}x =1\,,\qquad -J'\ge \frac{r}k\quad\mbox{on $[0,3/4]$}\,,\qquad\|J\|_{C^0(\R^d)}\le k\,,
\end{equation}
for some $k>0$. For subsets $E,F$ of $\R^d$, let $\de_J(E;F)$ be defined by (\ref{redef}).

If $E\subset F\subset \R^d$ are such that
\begin{equation}
  \label{tange}
 \frac14\le \Big(\frac{|F|}{|B|}\Big)^{1/d}-\Big(\frac{|E|}{|B|}\Big)^{1/d}\le \frac{3}4\,,\qquad |E|\ge 2\,|B|\,,
\end{equation}
then one has
\begin{equation}
  \label{quantitative riesz sets}
  |E|^{1-1/d}\,\a(E;B)^{8(d+2)}\le C(d,k)\,\de_J(E;F)\,.
\end{equation}
\end{theorem}

\begin{remark} {\rm Note the factor of $|E|^{1-1/d}$ on the left-hand side of (\ref{quantitative riesz sets}), which is proportional
to $\H^{d-1}(\partial E^*)$. That is, the size of the ``remainder term'' is a multiple, depending on the asymmetry $\a(E;B)$,
of the perimeter of $E^*$. This is the size we would expect since we are measuring the deficit under rearrangement of
a finite range non-local perimeter functional.}
\end{remark}

\section{Improvement in the Brunn-Minkowski inequality} In this section we prove Theorem \ref{thm: BM 1/4}. Recall that $K$ is a bounded open convex set containing the origin, so that
\[
K=\{x\in\R^d:\|x\|<1\}\,,
\]
where $\|\cdot\|:\R^d\to[0,\infty)$ is the convex one-homogenous function on $\R^d$ defined as
\[
\|x\|=\inf\Big\{t>0\, : \frac{x}t\in K\Big\}\,,\qquad x\in\R^d\,.
\]
Let us consider the convex one-homogenous function $\|\cdot\|_*:\R^d\to[0,\infty)$ defined by setting
\[
\|y\|_*=\sup\Big\{x\cdot y:\|x\|<1\Big\}\,,\qquad y\in\R^d\,.
\]
Given a set of locally finite perimeter $E$ in $\R^d$, and a bounded open set $A\subset\R^d$, we set
\[
P_K(E;A)=\int_{A\cap\pa^*E}\|\nu_E\|_*\,{\rm d}\H^{d-1}\,,
\]
where $\pa^*E$ denotes the reduced boundary of $E$ and where $\nu_E$ is the measure theoretic outer unit normal to $E$ (see \cite[Chapter 16]{maggiBOOK} for these definitions). When $E$ is an open set with Lipschitz boundary, one can replace $\pa^*E$ with the topological boundary in this definition. Notice that because we do not assume that $K = -K$, it may well be that
$\|y\|_*\ne\|-y\|_*$, and thus that $P_K(E;A)\ne P_K(\R^d\setminus E;A)$.

Given this notion of anisotropic perimeter we can consider the isoperimetric problem of determining
\begin{equation}
  \label{anisotropic iso problem}
  \inf\Big\{P_K(E):|E|=m\Big\}\,,\qquad m>0\,.
\end{equation}
It turns out that if $r>0$ is such that $|r\,K|=m$, then $\{x+r\,K\}_{x\in\R^d}$ is the family of minimizers in \eqref{anisotropic iso problem}. By taking into account that
\begin{equation}
  \label{K perimeter of K}
  P_K(K)=d\,|K|\,,
\end{equation}
this assertion can in fact be reformulated as the so-called Wulff inequality \eqref{wulff inequality}. By repeating {\it verbatim} the classical proof of the coarea formula (see, for example, \cite[Theorem 13.1]{maggiBOOK}) we find that
\[
\int_A\,\|-\nabla u(x)\|_*\,{\rm d}x=\int_\R\,P_K(\{u>t\};A)\,{\rm d}t\,,
\]
(as elements of $[0,\infty]$) whenever $u:\R^d\to\R$ is a Lipschitz function and $A$ is an open set. If we use sub-level sets instead of super-level sets of $u$ we find of course
\begin{equation}
  \label{coarea formula}
  \int_A\,\|\nabla u(x)\|_*\,{\rm d}x=\int_\R\,P_K(\{u<t\};A)\,{\rm d}t\,.
\end{equation}
Setting $K_s=s\,K=\{s\,x:x\in K\}$, $s>0$, we now prove the following lemma.

\begin{lemma}\label{lemma: coarea}
  If $E$ is a closed set in $\R^d$, then
  \begin{equation}
    \label{volume coarea}
      |E+K_r|=|E|+\int_0^r\,P_K(E+K_s)\,{\rm d}s\,.
  \end{equation}
\end{lemma}

\begin{proof}
  If we set
  \[
  g_E(x)=\inf\Big\{\|x-y\|:y\in E\Big\}\,,\qquad x\in\R^d\,,
  \]
  then $g_E$ is a Lipschitz function with
  \begin{equation}
    \label{gE1}
      |g_E(x)-g_E(y)|\le\|x-y\|\,,\qquad\forall x,y\in\R^d\,,
  \end{equation}
  and $\{g_E=0\}=E$, $\{g_E<s\}=E+K_s$ for every $s>0$. Let $x$ be a point of differentiability for $g_E$. By \eqref{gE1}, we certainly have
    \begin{equation}
    \label{gE2}
  |\nabla g_E(x)\cdot e|\le\|e\|\,,\qquad\forall e\ne 0\,,
  \end{equation}
  If now $g_E(x)>0$ then there exists $z\in E$ such that $g_E(x)=\|x-z\|$ and for $0<h<\|x-z\|$ and $e_0=-(x-z)/\|x-z\|$ we easily find
  \[
  g_E(x+h\,e_0)\le\|x+h\,e_0-z\|=\|x-z\|\Big(1-\frac{h}{\|x-z\|}\Big)=g_E(x)-h\,,
  \]
  that gives $\nabla g_E(x)\cdot e_0\le -1$, or, in other terms
  \begin{equation}
  \label{gE3}
  \nabla g_E(x)\cdot (-e_0)\ge1\,,\qquad\mbox{for some $e_0$ with $\|-e_0\|=1$}\,.
  \end{equation}
  Combining \eqref{gE2} and \eqref{gE3} with Rademacher's theorem we thus find that $\|\nabla g_E\|_*=1$ a.e. on $\{g_E>0\}$ so that, by the coarea formula \eqref{coarea formula} (applied to the open set $A=\{0<g_E<r\}$)
  \[
  |\{0<g_E<r\}|=\int_\R\,P_K\Big(\{g_E<s\};\{0<g_E<r\}\Big)\,{\rm d}s\,.
  \]
  Since we have
  \begin{eqnarray*}
    |\{0<g_E<r\}|&=&|E+K_r|-|E|\,,
    \\
    \int_\R\,P_K\Big(\{g_E<s\};\{0<g_E<r\}\Big)\,{\rm d}s
    &=&\int_0^r\,P_K(\{g_E<s\})\,{\rm d}s\,,
  \end{eqnarray*}
  the proof is complete.
\end{proof}

We shall also need the following elementary lemma.

\begin{lemma}\label{lemma: lip alpha}
  If $E, F\subset\R^d$ are Lebesgue measurable sets, with $0<|E|\,|F|<\infty$, then
  \begin{equation}
    \label{Fraenkel Lip}
      \Big||E|\a(E;K)-|F|\a(F;K)\Big|\le \left|E\,\Delta F\right|\,.
  \end{equation}
\end{lemma}

\begin{proof}
  Let $x\in\R^d$ be such that $2|E|\a(E;K)=|E\Delta (x+r_E\,K)|$, where $r_E=(|E|/|K|)^{1/d}$. If $r_F=(|F|/|K|)^{1/d}$, then we have,
  \begin{eqnarray*}
  2|F|\a(F;K)&\le& |F\Delta (x+r_F\,K)|
  \\
  &\le& |F\Delta E|+|(x+r_E\,K)\Delta (x+r_F\,K)|+2|E|\a(E;F)\,.
  \end{eqnarray*}
  Since $K$ is star-shaped with respect to the origin we have
  \[
  |(x+r_E\,K)\Delta (x+r_F\,K)|=||r_E\,K|-|r_F\,K||=||F|-|E||\le|E\Delta F|\,,
  \]
  and thus we conclude
  \[
  |F|\a(F;K)-|E|\a(E;K)\le |E\Delta F|\,.
  \]
  By symmetry, we find \eqref{Fraenkel Lip}.
\end{proof}

\begin{proof}[Proof of Theorem \ref{thm: BM 1/4}] {\it Step one:} We start showing that, in proving Theorem \ref{thm: BM 1/4}, we can directly assume that $E$ is a compact set. Indeed, let $E$ be a Lebesgue measurable set, and consider a sequence of compact sets $\{E_h\}_{h\in\N}$ with $E_h\subset E$, $|E_h|>0$, and $|E\setminus E_h|\to 0$ as $h\to\infty$. By Lemma \ref{lemma: lip alpha} we have $\a(E_h;K)\to\a(E;K)$ as $h\to\infty$, while the inclusion $E_h+K\subset E+K$ implies
\[
\de(E_h;K)\le \s(E_h;K)^{1/d}\Big\{\frac{|E+K|^{1/d}}{|E_h|^{1/d}+|K|^{1/d}}-1\Big\}\,,
\]
so that, in particular, $\limsup_{h\to\infty}\de(E_h;K)\le\de(E;K)$. Therefore, if Theorem \ref{thm: BM 1/4} holds true on compact sets, then it holds true on Lebesgue measurable sets.

\bigskip

\noindent {\it Step two:} We address the one-dimensional case $d=1$. We want to prove that
\[
2\,\de(E;K)\ge\a(E;K)\,,
\]
where $K=(a,b)$ for some $a<0<b$ and where $E\subset\R$ is compact. Exploiting the scale invariance properties of $\de$ and $\a$ we can equivalently prove that
\begin{equation}\label{proof 14 n=1}
2\,\max\Big\{r,\frac1r\Big\}\,\Big(\frac{|E+K_r|}{|K+K_r|}-1\Big)\ge\a(E;K)\,,\qquad \forall r>0\,,
\end{equation}
where $E$ is a compact set with $|E|=|K|$. Let us now set
\[
\a=\|-1\|_*\,,\qquad\beta=\|1\|_*\,,
\]
so that $\a\,,\b>0$ and, if $\{(a_i,b_i)\}_{i=1}^m$ is a family of bounded open intervals in $\R$ lying at mutually positive distances, then
\[
P_K\Big(\bigcup_{i=1}^m\,(a_i,b_i)\Big)=m\,(\a+\beta)\,.
\]
Since $E+K_s$ is a bounded open set in $\R$ for every $s>0$, with $E+K_s\subset E+K_r$ if $s<r$, and since, by Lemma \ref{lemma: coarea},
\[
\infty>|E+K_r|=|E|+\int_0^r\,P_K(E+K_s)\,{\rm d}s\,,
\]
we deduce that $E+K_s$ is a finite union of intervals for every $s>0$. In particular, if we set
\[
N(r)=\frac{P_K(E+K_r)}{\a+\b}\,,\quad\quad r>0\,,
\]
then $N(r)\in\N$ for every $r>0$, $N(r)$ is decreasing on $r>0$, and $N(r)\ge 1$ for every $r>0$. Since $P(K+K_s)=1$ for every $s>0$, by \eqref{volume coarea} we find that
  \begin{equation}
  \label{deficit n=1 1}
  |E+K_r|-|K+K_r|=(\a+\b)\,\int_0^r\,(N(s)-1)\,{\rm d}s\,.
  \end{equation}
Let us now set
  \[
  r_0=\inf\{r>0:N(r)=1\}\,.
  \]
(Notice that, trivially, $r_0<\infty$.) If $r_0=0$, then $E+K_r$ is an interval for every $r>0$, thus $\a(E;K)=0$ and \eqref{proof 14 n=1} follows immediately. If $r_0>r$, then by \eqref{deficit n=1 1}, and since $\a(E;K)<1$, we find
  \begin{equation}
    \label{n=1 caso 1}
      |E+K_r|-|K+K_r|\ge (\a+\b)\,r\,(N(r)-1)\ge (\a+\b)\,r\ge (\a+\b)\,r\,\a(E;K)\,.
  \end{equation}
Since $\a+\b=P_K(K)=|K|$ (by \eqref{K perimeter of K}) and $|K+K_r|=(1+r)\,|K|$, we conclude from \eqref{n=1 caso 1}
\[
\frac{|E+K_r|}{|K+K_r|}-1\ge \frac{r}{1+r}\,\a(E;K)\,,\qquad (r\le r_0)\,,
\]
which is easily seen to imply \eqref{proof 14 n=1}. We are thus left to consider \eqref{proof 14 n=1} in the case that $r>r_0$. In this case from \eqref{deficit n=1 1}, the definition of $r_0$ and, again, by $\a+\b=|K|$, we find
  \begin{equation}
  \label{deficit n=1 2}
  |E+K_r|-|K+K_r|=(\a+\b)\,r_0=|K|\,r_0\,,
  \end{equation}
  as well as that
  \begin{equation}
    \label{r0 case two}
      |E+K_r|-|K+K_r|=|E+K_{r_0}|-|K+K_{r_0}|\,.
  \end{equation}
  Up to a translation, $E+K_{r_0}=(-R\,a,R\,b)=K_R$ for some $R>0$. Therefore, by \eqref{r0 case two},
  \begin{equation}
  \label{deficit n=1 3}
  |E+K_r|-|K+K_r|=|K|\,(R-(1+r_0))\,.
  \end{equation}
  Adding up \eqref{deficit n=1 2} and \eqref{deficit n=1 3}, and since $E\subset E+K_{r_0}=K_R$, we find
  \begin{eqnarray*}\nonumber
  2(|E+K_r|-|K+K_r|)&=&|K|\,(R-1)=|K_R\setminus K|\ge|E\setminus K|=\frac{|E\Delta K|}2
  \\
  &\ge&|K|\,\a(E;K)\,.
  \end{eqnarray*}
  that in turn gives
  \[
  \frac{|E+K_r|}{|K+K_r|}-1\ge\frac{\a(E;K)}{2(1+r)}\,,\qquad (r>r_0)\,.
  \]
  Since this last inequality implies \eqref{proof 14 n=1}, we have completed the proof of step two.

  \bigskip

  \noindent {\it Step three:} We now prove the theorem in dimension $d\ge 2$. By step one and by exploiting the scale invariance of $\de$ and $\a$, we need to prove that if $E$ is a compact set in $\R^d$ with $|E|=|K|$, then
  \begin{eqnarray}
    \label{brunn-minkowski inq improved Br 2p}
  \a(E;K)^4\le C(d)\,\max\{1,r^{4d+2}\}\,\de(E;K_r)\,,\qquad\forall r>0\,,
  \end{eqnarray}
  where
  \begin{eqnarray}\label{delta r}
  \de(E;K_r)=\max\Big\{r,\frac1r\Big\}\,\Big(\frac{|E+K_r|^{1/d}}{|K+K_r|^{1/d}}-1\Big)\,,\quad\quad r>0\,.
  \end{eqnarray}
  Let us thus fix a value of $r>0$, and set for the sake of brevity
  \[
  \eta=\frac{|E+K_r|}{|K+K_r|}\,.
  \]
  By \eqref{brunn-minkowski inq}, $\eta\ge 1$. We claim that we may directly assume
  \begin{equation}
    \label{proof 14 small deficit}
    \eta\le 1+ \kappa(r)\,,
  \end{equation}
  where
  \begin{equation}
    \label{kappa}
    \kappa(r)=\min\Big\{r,\frac1r\Big\}\,.
  \end{equation}
  Indeed $\kappa(r)\in(0,1]$ for every $r>0$ and
  \begin{equation}\label{useful inq}
  (1+\k)^{1/d}-1\ge (2^{1/d}-1)\,\k\,,\qquad\forall \k\in[0,1]\,.
  \end{equation}
  Therefore, if \eqref{proof 14 small deficit} does not hold true, then, as $\a(E;K)<1$,
  \begin{eqnarray*}
  \de(E;K_r)&=& \max\Big\{r,\frac1r\Big\}\,(\eta^{1/d}-1)\ge\max\Big\{r,\frac1r\Big\}\,\Big((1+\k(r))^{1/d}-1\Big)
  \\
  &\ge&(2^{1/d}-1)\,\max\Big\{r,\frac1r\Big\}\kappa(r)=(2^{1/d}-1)\ge
   (2^{1/d}-1)\,\a(E;K)^4\,,
  \end{eqnarray*}
  and \eqref{brunn-minkowski inq improved Br 2p} follows provided
  \begin{equation}
    \label{value of Cn condition1}
      C(d)\ge\frac1{2^{1/d}-1}\,.
  \end{equation}
  We have thus reduced to consider the case that \eqref{proof 14 small deficit} holds true. In this case, by \eqref{useful inq} we find that
  \begin{eqnarray}
  \de(E;K_r)\ge
    \max\Big\{r,\frac1{r}\Big\}\,(2^{1/d}-1)\,(\eta-1)\,.
    \label{lets see}
  \end{eqnarray}
  Having this lower bound for $\de(E;K_r)$ in mind, we now apply Lemma \ref{lemma: coarea} to find
  \begin{equation}
  \label{proof 14 1}
  |E+K_r|-|K+K_r|=\int_0^{r}\,\Big(P_K(E+K_s)-P_K(K+K_s)\Big)\,{\rm d}s\,.
  \end{equation}
  From now, for the sake of brevity, we directly set $\a(G;K)=\a(G)$ for every $G\subset\R^d$. By applying the quantitative Wulff inequality \eqref{quantitative isop inq} to $E+K_s$ we deduce that
  \begin{eqnarray}\label{k1}
    |E+K_r|-|K+K_r|
    &\ge&n|K|^{1/d}\int_0^{r}\,|E+K_s|^{1/d'}\frac{\a(E+K_s)^2}{C_0(d)}\,{\rm d}s
    \\\nonumber
    &&+n|K|^{1/d}\,\int_0^{r}\,\Big(|E+K_s|^{1/d'}-|K+K_s|^{1/d'}\Big)\,{\rm d}s\,,
  \end{eqnarray}
  where the second integral on the right-hand side of \eqref{k1} is non-negative by the Brunn-Minkowski inequality. By H\"older inequality, we thus find
  \begin{eqnarray}\nonumber
    &&\frac{C_0(d)}{d\,|K|^{1/d}}\,\Big(|E+K_r|-|K+K_r|\Big)\,\int_0^{r}\,|E+K_s|^{1/d'}\,{\rm d}s
    \\\nonumber
    &&\qquad\ge \left(\int_0^{r}\,|E+K_s|^{1/d'} \a(E+K_s)\,{\rm d}s\right)^2
    \\\label{k7}
    &&\qquad\ge |E+K_r|^{-2/d}\,\left(\int_0^{r}\,|E+K_s|\a(E+K_s)\,{\rm d}s\right)^2\,.
  \end{eqnarray}
  Now, by Wulff's inequality \eqref{wulff inequality}, by Lemma \ref{lemma: coarea}, and by \eqref{proof 14 small deficit}
  \begin{eqnarray}\nonumber
  n|K|^{1/d}\,\int_0^{r}\,|E+K_s|^{1/d'}\,{\rm d}s&\le&\int_0^r P(E+K_s)\,{\rm d}s=|E+K_r|-|E|
  \\\nonumber
  &\le&\eta\,|K+K_r|-|K|
  \\\nonumber
  &\le&|K|\Big((1+\k(r))(1+r)^d-1\Big)
  \\\label{k6}
  &\le& 2^{d+1}\,|K|\,\max\{r,r^d\}\,\,;
  \end{eqnarray}
  in particular, having shown that $|E+K_r|-|E|\le2^{d+1}\,|K|\,\max\{r,r^d\}$, we certainly have
  \begin{equation}
    \label{k4}
      |E+K_r|\le 2^{d+2}\,|K|\,\max\{1,r^d\}\,.
  \end{equation}
  Thus, by \eqref{k7}, \eqref{k6} and \eqref{k4}, we find that
  \begin{eqnarray}\nonumber
  &&\Big(\int_0^{r}\,|E+K_s|\a(E+K_s)\,{\rm d}s\Big)^2
  \\\nonumber
  &&\hspace{1cm}\le \frac{C_0(d)}{d\,|K|^{1/d}}\,\Big(|E+K_r|-|K+K_r|\Big)\,\frac{2^{d+1}\,|K|\,\max\{r,r^d\}}{d|K|^{1/d}}\times
  \\\nonumber
  &&\hspace{1cm}\times\Big(2^{d+2}\,|K|\,\max\{1,r^d\}\Big)^{2/d}
  \\\nonumber
  &&\hspace{1cm}=\frac{2^{d+3+4/d}\,C_0(d)}{d^2}\,\Big(|E+K_r|-|K+K_r|\Big)\,|K|\,\max\{r,r^d\}\,\max\{1,r^2\}
  \\\nonumber
  &&\hspace{1cm}=\frac{2^{d+3+4/d}\,C_0(d)}{d^2}\,\Big(\eta-1\Big)\,|K|^2\,(1+r)^d\,\max\{r,r^d\}\,\max\{1,r^2\}
  \\\label{k5}
  &&\hspace{1cm}\le\frac{2^{5\,d}\,C_0(d)}{d^2}\,|K|^2\,\max\{r,r^{2(d+1)}\}\,\Big(\eta-1\Big)\,,
  \end{eqnarray}
  where in the last inequality we have used $(1+r)^d\le 2^d\max\{1,r^d\}$ and $2d+3+(4/d)\le 5d$.
  Let us now consider $\e\in(0,\min\{r,1\})$, and apply Lemma \ref{lemma: lip alpha} to compare $E$ and $E+K_s$ for $s\in(0,\e)$. In this way we find that
  \begin{equation}
    \label{k2}
      \int_0^\e |E+K_s|\alpha(E+K_s)\,{\rm d}s\ge \e\,|K|\a(E)-\int_0^\e|E\Delta(E+K_s)|\,{\rm d}s\,,
  \end{equation}
  where
  \begin{eqnarray}\nonumber
  &&\int_0^\e|E\Delta (E+K_s)|\,{\rm d}s=\int_0^\e\Big(|E+K_s|-|E|\Big)\,{\rm d}s
  =\int_0^\e \,{\rm d}s\int_0^s \,P_K(E+K_t)\,{\rm d}t
  \\\nonumber
  &=&\int_0^\e \,{\rm d}s\int_0^s \Big(P_K(E+K_t)-P_K(K+K_t)\Big)\,{\rm d}t+d|K|\int_0^\e \,{\rm d}s\int_0^s (1+t)^{d-1}\,{\rm d}t
  \\\nonumber
  &\le&\e\,\Big(|E+K_r|-|K+K_r|\Big)+|K|\Big(\frac{(1+\e)^{d+1}}{d+1}-\frac1{d+1}-\e\Big)
  \\\nonumber
  &\le&\e\,\Big(|E+K_r|-|K+K_r|\Big)+d\,2^{d-1}\,|K|\,\e^2
  \\\nonumber
  &=&\e\,(1+r)^d\,|K|\,(\eta-1)+d\,2^{d-1}\,|K|\,\e^2
  \\\label{k3}
  &\le&\e\,2^d\,\max\{1,r^d\}\,|K|\,(\eta-1)+d\,2^{d-1}\,|K|\,\e^2
  \end{eqnarray}
  where we have also used the elementary inequality
  \[
  \frac{(1+x)^{d+1}}{d+1}-\frac1{d+1}-x\le d\,2^{d-1}\,x^2\,,\qquad\forall x\in[0,1]\,.
  \]
  We now combine \eqref{k5}, \eqref{k2}, and \eqref{k3} to prove that
  \begin{eqnarray}\label{1/4 estimate step two 4}
  \a(E)\le a\,\max\{r^{1/2},r^{d+1}\}\,\frac{\sqrt{\eta-1}}\e+2^d\,\max\{1,r^d\}\,(\eta-1)+b\,\e\,,
  \end{eqnarray}
  for every $\e\in(0,\min\{1,r\})$, where we have set
  \[
  a=\frac{\sqrt{2^{5\,d}\,C_0(d)}}{d}\,\,,\qquad b=d\,2^{d-1}\,.
  \]
  In the case $r<1$, by \eqref{kappa}, we have $\eta-1\le r$, and thus
  \[
  \e=\Big(\frac{\eta-1}{r}\Big)^{1/4}r\,,
  \]
  is an admissible choice in \eqref{1/4 estimate step two 4}; correspondingly we find
  \begin{eqnarray*}
  \a(E)&\le&  a\,r^{1/2}\,\frac{(\eta-1)^{1/4}}{r^{3/4}}+2^d\,(\eta-1)+b\,\Big(\frac{\eta-1}{r}\Big)^{1/4}r
  \\
  &\le&  (a+b+2^d\,r)\,\Big(\frac{\eta-1}{r}\Big)^{1/4}\,.
  \end{eqnarray*}
  Since, by \eqref{lets see}, $\de(E;K_r)\ge(\log(2)/d)\,((\eta-1)/r))$ when $r<1$, we conclude that
  \[
  \a(E)^4\le(a+b+2^d)^4\frac{d}{\log(2)}\,\de(E;K_r)\,,\qquad\mbox{if $r\le 1$}\,.
  \]
  This implies \eqref{brunn-minkowski inq improved Br 2p} for every $r\le 1$, provided we set
  \begin{equation}
    \label{a value of Cn}
      C(d)=\frac{d}{\log(2)}\,\Big(\frac{\sqrt{2^{5\,d}\,C_0(d)}}{d}+d\,2^{d-1}+2^d\Big)^4\,.
  \end{equation}
  (Notice that this value of $C(d)$ satisfies \eqref{value of Cn condition1}.) If, instead, $r>1$, then by \eqref{kappa} we have $r(\eta-1)\le 1$, and
  \[
  \e=(r(\eta-1))^{1/4}\,,
  \]
  is admissible in \eqref{1/4 estimate step two 4}, that gives
  \begin{eqnarray*}
  \a(E)&\le&a\,r^{d+(3/4)}\,(\eta-1)^{1/4}+2^d\,r^d\,(\eta-1)+b\,r^{1/4}\,(\eta-1)^{1/4}
  \\
  &\le& \Big(a\,r^{d+(3/4)}\,+2^d\,r^{d-(3/4)}+b r^{1/4}\,\Big)\,(\eta-1)^{1/4}\,.
  \end{eqnarray*}
  At the same time, by \eqref{lets see} we have $\de(E;K_r)\ge (\log(2)/d)\,r\,(\eta-1)$, so that
  \begin{eqnarray*}
  \a(E)^4&\le& \Big(a\,r^{d+(3/4)}\,+2^d\,r^{d-(3/4)}+b\,r^{1/4}\Big)^4\,(\eta-1)
  \\
  &\le&\frac{d}{\log(2)}\,\Big(a+b+2^d\Big)^4\,r^{4d+3}\,\frac{\de(E;K_r)}{r}\,.
  \end{eqnarray*}
  This concludes the proof of \eqref{brunn-minkowski inq improved Br 2p}.
\end{proof}

\section{Improvement in the Gaussian concentration inequality}

This section is devoted to the proof of Theorem \ref{thm gaussian concentration}. As in the case of the proof of Theorem \ref{thm: BM 1/4}, we shall need two preliminary facts: first, if $E\subset\R^d$ is closed, then
\begin{equation}
  \label{gaussian coarea}
\g_d(E+rB)-\g_d(E)=\frac1{\sqrt{2\pi}}\,\int_0^r\,P_\g(E+B_t)\,dt\,;
\end{equation}
second, if $E, F\subset\R^d$, then
\begin{equation}
  \label{gaussian lip}
  |\a_\g(E)-\a_\g(F)|\le 2\,\g_d(E\Delta F)\,.
\end{equation}
Since the proofs are entirely analogous to the arguments of Lemma \ref{lemma: coarea} and Lemma \ref{lemma: lip alpha} we omit them. We notice that, since $\a_\g(E)\le 1$ for every $E\subset\R^d$, then \eqref{gaussian lip} immediately implies
\begin{equation}
  \label{gaussian lip x}
  |\a_\g(E)^2-\a_\g(F)^2|\le 2|\a_\g(E)-\a_\g(F)|\le 4\,\g_d(E\Delta F)\,.
\end{equation}
It will be convenient to set $\s_E:[0,\infty)\to[s_E,\infty)$,
\begin{equation}
  \label{montpellier2}
  \g_d(H(\s_E(t))=\g_d(E+B_t)=\g_d(H(s_{E+B_t}))\,,\qquad t\ge 0\,,
\end{equation}
i.e. $\s_E(t)=s_{E+B_t}$, and, in particular, $\s_E(0)=s_E$. It is useful to keep in mind that since $\phi(s)$ is increasing, see \eqref{phi definition}, and since $\g_d(E+B_t)\ge\g_d(H(s_E)+B_t)$ with $H(s_E)+B_t=H(s_E+t)$, we clearly have that
\[
\s_E(t)\ge s_E+t\,,\qquad\forall t>0\,.
\]
However, taking into account that
\begin{equation}
  \label{gaussian perimeter halfspace s}
  P_\g(H(s))=e^{-s^2/2}\,,\qquad\forall s\in\R\,,
\end{equation}
we easily see that $P_\g(H(s_{E+B_t}))-P_\g(H(s_E)+B_t)$ has no definite sign, and that
\begin{equation}
  \label{noticed}
  P_\g(H(s_{E+B_t}))-P_\g(H(s_E)+B_t)\ge0\qquad\mbox{if and only if}\qquad \s_E(t)\le|s_E+t|\,.
\end{equation}

\begin{proof}
  [Proof of Theorem \ref{thm gaussian concentration}] We fix $\l\in(\g_d(E),1)$ and $r<r_E(\l)$. By an approximation argument we may directly assume that $E$ is closed, and since $\a_\g(E)\le 1$ and $C_*\ge 1$, we can definitely assume that
  \[
  \de_\g^r(E)\le 1\,.
  \]
  Next we exploit \eqref{gaussian coarea} to deduce \eqref{gauss intro}, which combined with \eqref{montpellier} and \eqref{gaussian perimeter halfspace s} gives, for every $r>0$,
  \begin{equation}
    \label{montpellier4}
  \sqrt{2\pi}\,\de_\g^r(E)+\int_0^r\,e^{-(s_E+t)^2/2}-e^{-\s_E(t)^2/2}\,dt\ge\int_0^r\,P_\g(E+B_t)-P_\g(H(\s_E(t)))\,dt\,.
  \end{equation}
  As noticed in \eqref{noticed}, the integral on the left-hand side could be positive depending on the value of $\g_d(E)$ and $t$. To estimate its size, we shall use the fact that
  \begin{equation}
    \label{montpellier3}
      |e^{-b^2/2}-e^{-a^2/2}|\le \sqrt{2\pi}\,\max\{a,b\}\,|\phi(a)-\phi(b)|\,,\qquad\forall a,b>0\,.
  \end{equation}
  where $\phi$ is defined as in \eqref{phi definition}. The proof of \eqref{montpellier3} is immediate: if we set
  \[
  \a=e^{-a^2/2}\in(0,1)\,,\qquad a=\sqrt{\log\Big(\frac1{\a^2}\Big)}\in(0,\infty)\,,\qquad\psi(\a)=\phi\Big(\sqrt{\log\Big(\frac1{\a^2}\Big)}\Big)
  \]
  and, similarly, $\b=e^{-b^2/2}$, then by a simple computation
  ${\displaystyle \psi'(\a)=\frac{-1}{\sqrt{2\pi\,\log(\a^{-2})}}}$
  and thus
  \[
  |\psi(\beta)-\psi(\a)|\ge\frac{|\beta-\a|}{\sqrt{2\pi\,\log(\min\{\a,\b\}^{-2})}}\,,\qquad\forall \a,\b\in(0,1)\,,
  \]
  which immediately gives us \eqref{montpellier3}. If $t$ is such that $\s_E(t)\le|s_E+t|$, then $e^{-(s_E+t)^2/2}-e^{-\s_E(t)^2/2}\le 0$. Otherwise,  by \eqref{montpellier3} we find
  \[
  e^{-(s_E+t)^2/2}-e^{-\s_E(t)^2/2}\le \sqrt{2\pi}\,\s_E(t)\,\de_\g^t(E)\,,
  \]
  and since $\phi(\s_E(t))=\g_d(E+B_t)\le\g_d(E+rB)<\l$ thanks to $r<r_E(\l)$, we conclude that
  \[
  e^{-(s_E+t)^2/2}-e^{-\s_E(t)^2/2}\le\,\sqrt{2\pi}\,\phi^{-1}(\l)\,\de_\g^t(E)\,.
  \]
  By \eqref{montpellier4} we thus infer
  \begin{eqnarray*}
  \sqrt{2\pi}\,\Big(1+\phi^{-1}(\l)\Big)\de_\g^r(E)&\ge&\int_0^r\,P_\g(E+B_t)-P_\g(H(\s_E(t)))\,dt
  \\
  &\ge&\int_0^r  \frac{e^{\s_E(t)^2/2}}{c\,(1+\s_E(t)^2)}\,\a_\g(E+B_t)^2\,dt\,,
  \end{eqnarray*}
  where in the last inequality we have used \eqref{gauss improved1}.  By exploiting the trivial estimate
  \[
  \frac{e^{s^2/2}}{1+s^2}\ge \frac{e^{s^2/4}}4\,,\qquad s>0\,,
  \]
  together with \eqref{gaussian lip x}, we find that for every $\rho\le r$
  \begin{eqnarray*}
  \sqrt{2\pi}\,\Big(1+\phi^{-1}(\l)\Big)\de_\g^r(E)\ge\frac{e^{s_E^2/4}}{4\,c}\int_0^\rho \,\Big(\a_\g(E)^2-4\,\g_d(E\Delta (E+B_t))\Big)\,dt\,.
  \end{eqnarray*}
  Now, since $\g_d(H(s_E)+B_t)-\g_d(E)=\phi(s_E+t)-\phi(s_E)\le t/\sqrt{2\pi}$, one gets
  \begin{eqnarray*}
  \int_0^\rho\,\g_d(E\Delta (E+B_t))\,dt&\le& \frac{\rho}{\max\{1,1/r\}}\,\de_\g^r\,(E)+\int_0^\rho\,\g_d(H(s_E)+B_t)-\g_d(E)\,dt
  \\
  &\le&\frac{\rho}{\max\{1,1/r\}}\,\de_\g^r(E)+\frac{\rho^2}{2\sqrt{2\pi}}\,,
  \end{eqnarray*}
  so that, in conclusion, for $\rho\le r<r_E(\l)$,
  \begin{equation}
    \label{montpellier5}
      \a_\g(E)^2\le 4\,\sqrt{2\pi}\,c\,e^{-s_E^2/2}\,(1+\phi^{-1}(\l))\,\frac{\de_\g^r(E)}{\rho\,\max\{1,1/r\}}+4\,\de_\g^r(E)+\frac{\rho}{2\sqrt{2\pi}}\,.
  \end{equation}
  If $r>\de_\g^r(E)^{1/2}$, then we choose $\rho= \de_\g^r(E)^{1/2}$ and thus obtain from \eqref{montpellier5} and $\de_\g^r(E)\le 1$
  \[
  \a_\g(E)^2\le \Big(4\,\sqrt{2\pi}\,c\,e^{-s_E^2/2}\,(1+\phi^{-1}(\l))+4+\frac{1}{2\sqrt{2\pi}}\Big)\,\sqrt{\de_\g^r(E)}\,;
  \]
  in, instead, $r\le\de_\g^r(E)^{1/2}$, then $r\le 1$ and setting $\rho=r$ we obtain
  \[
  \a_\g(E)^2\le \Big(4\,\sqrt{2\pi}\,c\,e^{-s_E^2/2}\,(1+\phi^{-1}(\l))+4\Big)\,\de_\g^r(E)+\frac{\sqrt{\de_\g^r(E)}}{2\sqrt{2\pi}}\,.
  \]
  By taking into account that $c=80\,\pi^2\,\sqrt{2\pi}$, one finds
  \[
  4\,\sqrt{2\pi}\,c\,e^{-s_E^2/2}\,(1+\phi^{-1}(\l))+4+\frac{1}{2\sqrt{2\pi}}\le (5+1280\,\pi^3)\,(1+\phi^{-1}(\l))\,,
  \]
  and  the proof of \eqref{gaussian concentration quantitative} is complete.
\end{proof}

\section{Stability in the Riesz rearrangement inequality} The goal of this section is proving Theorem \ref{reiszst1}. Let us recall that we are considering a decreasing Lipschitz function $J:[0,\infty)\to[0,\infty)$  with  ${\rm spt}(J) \subset [0,1]$ such that
\begin{equation}
  \label{hp on J}
  \int_{\R^d}J(|x|){\rm d}x =1\,,\qquad -J'\ge \frac{r}k\quad\mbox{on $[0,3/4]$}\,,\qquad\|J\|_{C^0(\R^d)}\le k\,,
\end{equation}
for some $k>0$, and that  given $E,F\subset\R^d$, we set
\begin{eqnarray*}
\mathcal{E}_J(E,F)&=&\int_{F}\int_{E}J(|x-y|)\,\dd x \dd y\,,
\\
\delta_J(E,F)&=& \mathcal{E}_J(E^*,F^*) - \mathcal{E}_J(E,F) \,.
\end{eqnarray*}
We shall actually assume that $E\subset F\subset \R^d$, and denote by
\[
r_{E,F}=\frac{|F|^{1/d}}{|B|^{1/d}} -\frac{|E|^{1/d}}{|B|^{1/d}}\,,
\]
the radius such that $|E^*+B_{r_{E,F}}|=|F^*|$. We assume that
\begin{equation}
  \label{hp on E and F}
 \frac14\le r_{E,F}\le \frac{3}4\,,\qquad |E|\ge 2\,|B|\,,
\end{equation}
and aim to prove
\[
|E|^{1-1/d}\,\a(E;B)^{8(d+2)}\le C(d,k)\,\de_J(E;F)\,.
\]

\begin{proof}[Proof of Theorem \ref{reiszst1}] {\it Step one}: Given $\l,\tau>0$ we set
\begin{equation}\label{Efulldef}
E^{\lambda,\tau} = E \backslash D^{\lambda,\tau} \ ,
\end{equation}
\begin{equation}\label{Ebaddef}
D^{\lambda,\tau} = \Big\{ \ x\in E\ :\ \frac{|E\cap B_{x,\tau}|}{|B_{x,\tau}|} < \lambda \Big\}\ .
\end{equation}
We claim that for every $\l>0$ and $\tau\in(0,r_{E,F})$ one has
\begin{equation}
  \label{step one}
  k\,\delta_J(E,F) \geq  \lambda\tau^{d+1}  \int_\tau^{r_{E,F}} |(E^{\lambda,\tau}+ ({r-\tau})B)\cap F^c| {\rm d}r\, .
\end{equation}
(Later on we shall specify the size of $\l$ and $\tau$, and they both will be small in terms of $\a(E;B)$.) To prove \eqref{step one}, since $r_{E,F}$ is  the difference of the radii of $F^*$ and $E^*$, one has, for every $r<r_{E,F}$,
\begin{eqnarray}
  \nonumber
    \int_{\R^d}  (1_{rB}\star 1_{E^*}) \,1_{F^*}&=&\int_{F^*}|E^*\cap B_{x,r}|\,\dd x=\int_{\R^d}|E^*\cap B_{x,r}|\,\dd x
    \\
  \label{rEF prop 1}
    &=&|E^*||rB|=|E||rB|=\int_{\R^d}1_E\star 1_{rB}\,.
\end{eqnarray}
By the layer-cake representation
${\displaystyle
J(|x|) = \int_{|x|}^\infty -J'(r)\dd r = \int_0^1   -J'(r)1_{rB}(x) \dd r}$
and by \eqref{rEF prop 1} we find
\begin{eqnarray}\label{defifo}
\delta_J(E,F) &=& \int_0^1 -J'(r)\left( \int_{\R^d} (1_{rB}\star 1_{E^*})\,1_{F^*} -
(1_{rB}\star 1_E )\,1_{F} \right) {\rm d}r \nonumber\\
&\geq & \int_0^{r_{E,F}} -J'(r)\left( \int_{\R^d} (1_{rB}\star 1_{E^*})\,1_{F^*} -
(1_{rB}\star 1_E )\,1_{F} \right)  {\rm d}r \nonumber\\\
&= & \int_0^{r_{E,F}} -J'(r)\left(\int_{\R^d}1_E\star 1_{rB} - \int_{\R^d}
1_{rB}\star 1_E (x)1_{F}(x){\rm d}x\right) {\rm d}r \nonumber\\
&\ge&
\frac{\tau}k\, \int_\tau^{r_{E,F}} \left(  \int_{F^c}
|E\cap B_{x,r}|\,\dd x\right) {\rm d}r\,,
\end{eqnarray}
where in the last inequality we have used \eqref{hp on J}. We now notice that
\begin{equation}\label{oax12}
|E\cap B_{x,r}| \geq  \lambda \tau^d|B|\,,\qquad\forall x\in E^{\lambda,\tau}+ B_{r-\tau}\,, \,\forall \tau< r\,.
\end{equation}
Indeed, by assumption on $x$, there exists $y\in B_{x,r-\tau}\cap E^{\lambda,\tau}$, so that, in particular, $B_{y,\tau}\subset B_{x,r}$, and thus
$y\in E^{\l,\tau}$  implies
$
|B_{x,r}\cap E| \geq|B_{y,\tau}\cap E| \geq \lambda\,|B|\,\tau^d\,.
$
By combining (\ref{defifo}) with \eqref{oax12} we thus find the lower bound \eqref{step one}.

\bigskip

\noindent {\it Step two}: We notice that the volumes of $|E|$ and $|F|$ differ by a ``surface term'',
\begin{equation}\label{relclo}
|F| - |E| \leq  C(d)\,|E|^{1-1/d}\,.
\end{equation}
Indeed, by definition of $r_{E,F}$ and by \eqref{hp on E and F} we have
\begin{eqnarray*}
|F|-|E|&\le&\Big(|E|^{1/d}+\frac34|B|^{1/d}\Big)^d-|E|=d\,\int_0^{3|B|^{1/d}/4}\big(|E|^{1/d}+t\big)^{d-1}\,dt
\\
&\le& C(d)\,\big(|E|^{1/d}+|B|^{1/d}\big)^{d-1}\le C(d)\,|E|^{1-1/d}\,.
\end{eqnarray*}

\bigskip

\noindent {\it Step three}: Given $\tau\in(0,r_{E,F})$, let us set
${\displaystyle
\ell=\int_{\tau/4 \leq |x| \leq \tau/2}J(|x|)\dd x}$,
(We shall pick $\tau$ so that $\ell$ will be small in terms of $\a(E;B)$.) We claim that if
\begin{equation}
  \label{hp mars}
  \delta_J(E,F) \leq  \ell^2 |B|^{1/d}|E|^{1-1/d}\,,
\end{equation}
and $\l$ is small enough in terms of $d$, then
\begin{equation}\label{inten9}
|D^{\lambda,\tau}| \leq  C(d)\,(\lambda + \ell)\,|E|^{1-1/d}\,.
\end{equation}
To this end, let us consider the truncated kernel
\[
 J_1 (r) =
 \begin{cases} J(r)/\ell & r\in (\tau/4,\tau/2)\,,
 \\ \quad 0 & r\notin (\tau/4,\tau/2)\, ,\end{cases}
\]
and notice that
\begin{equation}
  \label{useful}
  \tau^d\,\int_0^\infty\,(-J_1'(r))\,\dd r\le C(d)\,.
\end{equation}
Indeed,
\[
\tau^d\,\int_0^\infty\,(-J_1'(r))\,\dd r\le C(d)\int_{\tau/4}^{\tau/2}\,(-J_1'(r))\,|rB|\,\dd r\le C(d)\,\int_{\R^d}\,J_1(|y|)\,\dd y= C(d)\,.
\]
By a  similar argument we find that
\begin{equation}
  \label{useful 2}
  J_1\star 1_{F^*}(x)\ge c(d)\,,\qquad \forall x\in F^*\,.
\end{equation}
To see this, notice that since $|F^*|\ge|E^*|=|E|\ge 2\,|B|$, one has
\[
|F^*\cap B_{x,r}|\ge c(d)\,|rB|,\qquad\forall x\in F^*\,,r<\frac34\,,
\]
and thus
\begin{eqnarray*}
J_1\star 1_{F^*}(x)&=&\int_{\tau/4}^{\tau/2}(-J_1'(r))\,|F^*\cap B_{x,r}|\,\dd r\ge
c(d)\,\int_{\tau/4}^{\tau/2}(-J_1'(r))\,|rB|\,\dd r
\\
&\ge&c(d)\,\int_{\R^d}\,J_1(|x|)\,\dd x=c(d)\,,
\end{eqnarray*}
as claimed. By \eqref{useful 2} we have
\begin{equation}
  \label{slow1}
  |D^{\l,\tau}|=|E^*\setminus(E^{\l,\tau})^*|\le C(d)\,\Big(\E_{J_1}(E^*,F^*)-\E_{J_1}((E^{\l,\tau})^*,F^*)\Big)\,.
\end{equation}
We now notice that thanks to \eqref{hp mars}
\begin{eqnarray*}
  \E_{J_1}(E^*,F^*)-\E_{J_1}(E,F)&=&\int_0^\infty(-J_1'(r))\,\dd r\int_{\R^d}(1_{E^*}\star 1_{rB})1_{F^*}-(1_{E}\star 1_{rB})1_{F}
  \\
  &\le&\frac1\ell \int_0^\infty(-J'(r))\,\dd r\int_{\R^d}(1_{E^*}\star 1_{rB})1_{F^*}-(1_{E}\star 1_{rB})1_{F}
  \\
  &\le&\frac{\de_J(E,F)}\ell \le \ell\,|B|^{1/d}\,|E|^{1-1/d}\,,
\end{eqnarray*}
while $\E_{J_1}(E^{\l,\tau},F)\le \E_{J_1}((E^{\l,\tau})^*,F^*)$ by Riesz inequality, so that \eqref{slow1} implies
\begin{eqnarray}
\nonumber
  |D^{\l,\tau}|&\le&C(d)\,\Big(\ell\,|E|^{1-1/d}+\E_{J_1}(E,F)-\E_{J}(E^{\l,\tau},F)\Big)
  \\\label{slow2}
  &=&C(d)\,\Big(\ell\,|E|^{1-1/d}+\E_{J_1}(D^{\l,\tau},F)\Big)\,.
\end{eqnarray}
Having in mind the decomposition $\E_{J_1}(D^{\l,\tau},F)=\E_{J_1}(D^{\l,\tau},E)+\E_{J_1}(D^{\l,\tau},F\setminus E)$, we first notice that
\begin{eqnarray}\nonumber
\E_{J_1}(D^{\l,\tau},E)&=&\int_{D^{\l,\tau}}\dd x\int_{\tau/4}^{\tau/2}(-J_1'(r))\,|E\cap B_{x,r}|\,\dd r
\\\nonumber
&\le&\int_{D^{\l,\tau}}|E\cap B_{x,\tau}|\,\dd x\int_{\tau/4}^{\tau/2}(-J_1'(r))\,\,\dd r
\\\label{useful 3}
&\le&\l\,|D^{\l,\tau}|\,\int_{\tau/4}^{\tau/2}(-J_1'(r))\,|B_{\tau}|\,\dd r\le C(d)\,\l\,|D^{\l,\tau}|\,,
\end{eqnarray}
where in the last inequality we have used \eqref{useful}. At the same time
\[
\E_{J_1}(D^{\l,\tau},F\setminus E)=\int_{F\setminus E}\dd x\int_{\tau/4}^{\tau/2}\,(-J_1'(r))\,|D^{\l,\tau}\cap B_{x,r}|\,\dd r\,,
\]
where, given $x\in F\setminus E$, either we have $D^{\l,\tau}\cap B_{x,r}=\emptyset$, or there exists $y\in D^{\l,\tau}\cap B_{x,r}$, in which case, by $r<\tau/2$, $B_{x,r}\subset B_{y,2r}\subset B_{y,\tau}$, and $y\in D^{\l,\tau}\subset E$, we obtain
\[
|D^{\l,\tau}\cap B_{x,r}|\le |E\cap B_{y,\tau}|\le\l\,|B_\tau|\,;
\]
we thus find, thanks to \eqref{useful} and \eqref{relclo}
\begin{equation}
  \label{useful 4}
  \E_{J_1}(D^{\l,\tau},F\setminus E)\le C(d)\,\l\,|F\setminus E|\le C(d)\,\l\,|E|^{1-1/d}\,.
\end{equation}
By combining \eqref{slow2}, \eqref{useful 3} and \eqref{useful 4} we thus find
\[
|D^{\l,\tau}|\le C(d)\,\Big((\ell+\l)\,|E|^{1-1/d}+\l\,|D^{\l,\tau}|\Big)\,.
\]
In particular, if $\l$ is small enough depending on $d$, obtain \eqref{inten9}.

\bigskip

\noindent {\it Step four}: We complete the proof of the theorem. We start by choosing the values of $\tau$ and $\l$. For a small value of $a>0$ to be fixed in the argument, and for some $p\ge 4$, let us set
\begin{equation}
  \label{choice of lamdba and tau}
  \l=\tau=a\,\a(E;B)^p\le a\,.
\end{equation}
(Recall that $\a(E;B)\le 1$ by definition.) Since $r_{E,F}\ge 1/4$, we can definitely entail $\tau< r_{E,F}$, and thus infer from \eqref{step one} that
\begin{equation}
  \label{step one x}
  k\,\delta_J(E,F) \geq  \lambda\tau^{d+1}  \int_0^{r_{E,F}-\tau} |(E^{\lambda,\tau}+ B_s)\cap F^c| \dd s
\end{equation}
holds. Now, since $(\tau/4,\tau/2)\subset (0,3/4)$, by \eqref{hp on J} we find
\begin{eqnarray*}
  \ell=\int_{\tau/4}^{\tau/2}(-J'(r))\,\om_d\,r^d\,\dd r
  \ge \frac{\om_d}{k}\,\int_{\tau/4}^{\tau/2}\,r^{d+1}\,\dd r\ge\frac{\tau^{d+2}}{C(d,k)}
  =\frac{\a(E;B)^{p(d+2)}}{C(d,k,a)}\,.
\end{eqnarray*}
Hence, by step three, either
\begin{equation}
  \label{either}
  \delta_J(E,F) \ge  \ell^2 |B|^{1/d}|E|^{1-1/d}\ge \frac{|E|^{1-1/d}\,\a(E;B)^{2p(d+2)}}{C(d,k,a)}\,,
\end{equation}
or \eqref{hp mars} holds, and thus
\begin{equation}
  \label{stimetta x}
  |D^{\lambda,\tau}| \leq  C(d)\,(\lambda + \ell)\,|E|^{1-1/d}\,.
\end{equation}
Let us now notice that, provided $a$ is small enough in terms of $d$ and $k$,
\begin{eqnarray*}
\ell&=&\int_{\tau/4 \leq |x| \leq \tau/2}J(|x|)\dd x\le C(d)\,\|J\|_{C^0(\R^d)}\,\tau^d
\\
&\le&C(d,k)\,a^d\,\a(E,B)^{p\,d}
\le a\,\a(E,B)^p\,,
\end{eqnarray*}
so that \eqref{choice of lamdba and tau} and \eqref{stimetta x} give us
\begin{equation}
  \label{dusty set small}
  |D^{\l,\tau}|\le C(d)\,a\,|E|^{1-1/d}\,\a(E;B)^p\,.
\end{equation}

Summarizing, either \eqref{either} holds, and then we are done, or the bad set $D^{\l,\tau}$ is actually small in terms of $\a(E;B)$. In this latter case we effectively exploit the lower bound \eqref{step one x} together with the quantitative Brunn-Minkowski inequality of Theorem \ref{thm: BM 1/4} in order to infer an estimate similar to \eqref{either}.

The argument goes as follows. By applying Theorem \ref{thm: BM 1/4} to $E^{\l,\tau}$ and $B_s$ with $s\in(0,r_{E,F}-\tau)$, we find that
\begin{equation}
  \label{gives}
  \frac{\a(E^{\l,\tau};B)^4}{C(d)}\le \max\Big\{\frac{|E^{\l,\tau}|}{|B_s|},\frac{|B_s|}{|E^{\l,\tau}|}\Big\}^{1/d}\,
\Big\{\Big(\frac{|E^{\l,\tau}+B_s|}{|(E^{\l,\tau})^*+B_s|}\Big)^{1/d}-1\Big\}\,.
\end{equation}
By \eqref{dusty set small}, $|E|\ge 2|B|$, and provided $a$ is small enough in terms of $d$,
\[
|E^{\l,\tau}|\ge |E|\Big(1-\frac{C(d)\,a}{|E|^{1/d}}\Big)\ge 2|B|\Big(1-\frac{C(d)\,a}{|B|^{1/d}}\Big)\ge |B|\,,
\]
so that $|E^{\l,\tau}|\ge |B_s|$ for $s\in(0,r_{E,F}-\tau)$ and \eqref{gives} gives us
\begin{equation}
  \label{gives2}
    \a(E^{\l,\tau};B)^4\le C(d)\,\frac{|E^{\l,\tau}|^{1/d}}{s}\,
  \Big(\frac{|E^{\l,\tau}+B_s|}{|(E^{\l,\tau})^*+B_s|}-1\Big)\,,
\end{equation}
where we have also used the concavity of $\eta\mapsto\eta^{1/d}$. We notice that by \eqref{stimetta x} and by $|E|\ge 2|B|$, if $a$ is small enough depending on $d$, then
\begin{equation}
\begin{split}
r_{E^{\lambda,\tau},F}-r_{E,F}&=\frac{|E|^{1/d}}{|B|^{1/d}}\,\left(1 - \Big(1-\frac{|D^{\lambda,\tau}|}{|E|}\Big)^{1/d}\right)
\\
&\le\frac{|E|^{1/d}}{|B|^{1/d}}\,\left(1 - \Big(1-\frac{C(d)\,a\,\a(E;B)^p}{|E|^{1/d}}\Big)^{1/d}\right)
\\
&\le C(d)\,a\,\a(E;B)^p\,.
\end{split}
\end{equation}
In particular,
\begin{equation}
  \label{in particular}
  r_{E,F}-\tau=r_{E,F}-a\,\a(E;B)^p> r_{E^{\l,\tau},F}-C_*(d)\,a\,\a(E;B)^p\,,
\end{equation}
for some specific constant $C_*(d)$. In particular, if we set
\[
I=\Big[r_{E^{\l,\tau},F}-2\,C_*(d)\,a\,\a(E;B)^p,r_{E^{\l,\tau},F}-C_*(d)\,a\,\a(E;B)^p\Big]\,,
\]
then for $a$ small enough
\begin{equation}
  \label{prop I}
  I\subset(0,r_{E,F}-\tau)\,,\qquad\mbox{with $\H^1(I)=C_*(d)\,a\,\a(E;B)^p$}\,;
\end{equation}
moreover, if $s\in I$, then $|(E^{\l,\tau})^*+B_{r_{E^{\l,\tau},F}}|=|F|$ gives
\[
|(E^{\l,\tau})^*+B_s|^{1/d}=|F|^{1/d}-(r_{E^{\l,\tau},F}-s)\,|B|^{1/d}\ge |F|^{1/d}-C(d)\,a\,\a(E;B)^p\,,
\]
that is (thanks to $|F|\ge |E|\ge 2|B|$)
\[
|(E^{\l,\tau})^*+B_s|\ge |F|\Big(1-C(d)\,a\,\a(E;B)^p\Big)\,,
\]
and thus
\[
|E^{\l,\tau}+B_s|-|(E^{\l,\tau})^*+B_s|\le|(E^{\l,\tau}+B_s)\setminus F|+C(d)\,a\,|F|\,\a(E;B)^p\,.
\]
By combining this inequality with \eqref{gives2} (and with $|(E^{\l,\tau})^*+B_s|\ge |F|/C(d)$)
\begin{eqnarray*}
\a(E^{\l,\tau};B)^4&\le& C(d)\,\frac{|E^{\l,\tau}|^{1/d}}{s\,|F|}\,
\Big(|(E^{\l,\tau}+B_s)\setminus F|+2\,|F|\,a^{1/4}\,\a(E;B)^p\Big)
\\
&\le& \frac{C(d)}s\,\,\frac{|E^{\l,\tau}|^{1/d}}{|F|}\,
|(E^{\l,\tau}+B_s)\setminus F|+\frac{C(d)}{s}\,\,a\,\a(E;B)^p\,.
\end{eqnarray*}
Of course $r_{E^{\l,\tau},F}\ge r_{E,F}\ge 1/4$ so that if $s\in I$, then $s\ge 1/8$, and thus we conclude by $p\ge 4$ and for $a$ small enough in terms of $d$ that
\[
\a(E^{\l,\tau};B)^p\le C(d)\,\,\frac{|E^{\l,\tau}|^{1/d}}{|F|}\,
|(E^{\l,\tau}+B_s)\setminus F|\,.
\]
On the one hand by \eqref{elementary} and by $|E|\ge 2|B|$
\[
|\a(E^{\l,\tau};B)-\a(E;B)|\le \frac{2\,|D^{\l,\tau}|}{|E|}\le  C(d)\,a\,\a(E;B)^p\,,
\]
so that
\[
\a(E^{\l,\tau};B)^p\ge \Big(\a(E;B)- C(d)\,a\,\a(E;B)^p\Big)^p\ge\frac{\a(E;B)^p}2\,,
\]
while on the other hand $|E^{\l,\tau}|^{1/d}|F|^{-1}\le |E|^{(1/d)-1}$ and thus
\[
|E|^{1-1/d}\,\a(E;B)^p\le C(d)\,|(E^{\l,\tau}+B_s)\setminus F|\,,\qquad\forall s\in I\,.
\]
By \eqref{step one x}, \eqref{prop I}, and the choices of $\l$ and $\tau$ we thus find
\[
k\,\delta_J(E,F) \geq  \frac{|E|^{1-1/d}}{C(d,a)}\,\lambda\,\tau^{d+1} \a(E;B)^{2\,p}=\frac{|E|^{1-1/d}}{C(d,a)}\,\a(E;B)^{(d+4)\,p}\,.
\]
By \eqref{either}, setting $p=4$ and recalling that $a=a(d,k)$ we deduce that
\[
|E|^{1-1/d}\,\min\{\a(E;B)^{4(d+4)},\a(E;B)^{8(d+2)}\}\le C(d,k)\,\de_J(E;F)\,,
\]
where the left-side is actually equal to $|E|^{1-1/d}\,\a(E;B)^{8(d+2)}$ as $\a(E;B)\le 1$.
\end{proof}

\bibliography{references}

\newcommand{\etalchar}[1]{$^{#1}$}
\begin{thebibliography}{CFMP11}

\bibitem[BBJ14]{barchiesibrancolinijulin}
A.~Brancolini, M.~Barchiesi, and V.~Julin.
\newblock Sharp dimension free quantitative estimates for the {G}aussian
  isoperimetric inequality.
\newblock 2014.
\newblock http://cvgmt.sns.it/paper/2516/.

\bibitem[Bur96]{Burchard}
A.~Burchard.
\newblock Cases of equality in the {R}iesz rearrangement inequality.
\newblock {\em Ann. of Math. (2)}, 143\penalty0 (3):\penalty0 499--527, 1996.

\bibitem[CCE{\etalchar{+}}09]{CarlenCELM}
E.~A. Carlen, M.~C. Carvalho, R.~Esposito, J.~L. Lebowitz, and R.~Marra.
\newblock Droplet minimizers for the {G}ates-{L}ebowitz-{P}enrose free energy
  functional.
\newblock {\em Nonlinearity}, 22\penalty0 (12):\penalty0 2919--2952, 2009.

\bibitem[CFMP11]{cianchifuscomaggipratelliGAUSS}
A.~Cianchi, N.~Fusco, F.~Maggi, and A.~Pratelli.
\newblock On the isoperimetric deficit in {G}auss space.
\newblock {\em Amer. J. Math.}, 133\penalty0 (1):\penalty0 131--186, 2011.

\bibitem[FJ13a]{figallijerison1}
A.~Figalli and D.~Jerison.
\newblock Quantitative stability for sumsets in $\mathbb{R}^n$.
\newblock 2013.
\newblock Preprint.

\bibitem[FJ13b]{figallijerison2}
A.~Figalli and D.~Jerison.
\newblock Quantitative stability of the {B}runn-{M}inkowski inequality.
\newblock 2013.
\newblock Preprint.

\bibitem[FMP08]{fuscomaggipratelli}
N.~Fusco, F.~Maggi, and A.~Pratelli.
\newblock The sharp quantitative isoperimetric inequality.
\newblock {\em Ann. Math.}, 168:\penalty0 941--980, 2008.

\bibitem[FMP09]{figallimaggipratelliBrunnMink}
A.~Figalli, F.~Maggi, and A.~Pratelli.
\newblock A refined {B}runn-{M}inkowski inequality for convex sets.
\newblock {\em Ann. {I}nst. {H}. {P}oincar\'e {A}nal. {N}on {L}in\'eaire},
  26\penalty0 (6):\penalty0 2511--2519, 2009.

\bibitem[FMP10]{FigalliMaggiPratelliINVENTIONES}
A~Figalli, F.~Maggi, and A.~Pratelli.
\newblock A mass transportation approach to quantitative isoperimetric
  inequalities.
\newblock {\em Inv. Math.}, 182\penalty0 (1):\penalty0 167--211, 2010.

\bibitem[GP69]{gatespenrose}
D.~J. Gates and O.~Penrose.
\newblock The van der {W}aals limit for classical systems. {I}. {A} variational
  principle.
\newblock {\em Comm. Math. Phys.}, 15:\penalty0 255--276, 1969.

\bibitem[HLP34]{HLPBOOK}
G.~H. Hardy, J.~E. Littlewood, and G.~P\'olya.
\newblock {\em Inequalities}.
\newblock Cambridge University Press, Cambridge, 1934.

\bibitem[HM53]{henstockmacbeath}
R.~Henstock and A.~M. Macbeath.
\newblock On the measure of sum-sets. i. the theorems of {B}runn, {M}inkowski,
  and {L}usternik.
\newblock {\em Proc. London Math. Soc. (3)}, 3:\penalty0 182--194, 1953.

\bibitem[Lie77]{LiebChoquard}
E.~H. Lieb.
\newblock Existence and uniqueness of the minimizing solution of {C}hoquard's
  nonlinear equation.
\newblock {\em Studies in Appl. Math.}, 57\penalty0 (2):\penalty0 93--105,
  1976/77.

\bibitem[LL01]{LiebLossBOOK}
E.~H. Lieb and M.~Loss.
\newblock {\em Analysis}, volume~14 of {\em Graduate Studies in Mathematics}.
\newblock American Mathematical Society, Providence, 2001.

\bibitem[LP66]{lebowitzpenrose}
J.~L. Lebowitz and O.~Penrose.
\newblock Rigorous treatment of the van der {W}aals-{M}axwell theory of the
  liquid-vapor transition.
\newblock {\em J. Mathematical Phys.}, 7:\penalty0 98--113, 1966.

\bibitem[Mag12]{maggiBOOK}
F.~Maggi.
\newblock {\em Sets of finite perimeter and geometric variational problems: an
  introduction to Geometric Measure Theory}, volume 135 of {\em Cambridge
  Studies in Advanced Mathematics}.
\newblock Cambridge University Press, 2012.

\bibitem[MN12]{mossel}
E.~Mossel and J.~Neeman.
\newblock Robust dimension free isoperimetry in {G}aussian space.
\newblock 2012.
\newblock arXiv:1202.4124v2.

\end{thebibliography}
\bibliographystyle{is-alpha}

\end{document}